\documentclass[11pt]{article}

\usepackage{amsmath}
\usepackage{amssymb}
\usepackage{amsthm}
\usepackage{graphics}
\usepackage[latin1]{inputenc}
\usepackage[english]{babel}
\usepackage[T1]{fontenc}
\usepackage{enumerate}

\newtheorem{thm}{Theorem}[section]
\newtheorem{prop}{Proposition}[section]

\newtheorem{lemma}{Lemma}[section]
\newtheorem{rem}{Remark}[section]
\newtheorem{rems}{Remarks}[section]

\newcommand{\R}{\mathbb{R}}             % REAL
\newcommand{\N}{\mathbb{N}}             % INTEGER
\newcommand{\Z}{\mathbb{Z}}             % ???
\newcommand{\C}{\mathbb{C}}             % COMPLEX
\newcommand{\e}{\epsilon}

\newcommand{\ds}{\displaystyle}

\newcommand{\Section}[1]{\section{#1} \setcounter{equation}{0}}

\begin{document}

\title{The anisotropic Calder\'{o}n problem on $3$-dimensional conformally St\"ackel manifolds}
\author{Thierry Daud\'e \footnote{Research supported by the ANR JCJC Horizons ANR-16-CE40-0012-01 and the \emph{Initiative d'excellence Paris-Seine}}, Niky Kamran \footnote{Research supported by NSERC grant RGPIN 105490-2018} and Francois Nicoleau \footnote{Research supported by the French GDR Dynqua}}

%%% TITLE %%%

%%% DATE %%%

%\date{\today}

%%% TITLE %%%

\maketitle

%%%%%%%%%%%%%%%%%%%%%%%%%%%%%%%%%%%%%% ABSTRACT %%%%%%%%%%%%%%%%%%%%%%%%%%%%%%%%%%%%%%%%%%%%%%%%%%%%%%%%%%%%%

\begin{abstract} 

Conformally St\"ackel manifolds can be characterized as the class of $n$-dimensional pseudo-Riemannian manifolds $(M,G)$ on which the Hamilton-Jacobi equation $G(\nabla u, \nabla u) = 0$ for null geodesics and the Laplace equation $-\Delta_G \, \psi = 0$ are solvable by R-separation of variables. In the particular case in which the metric has Riemannian signature, they provide explicit examples of metrics admitting a set of $n-1$ \emph{commuting conformal symmetry operators} for the Laplace-Beltrami operator $\Delta_G$. In this paper, we solve the anisotropic Calder\'on problem on compact $3$-dimensional Riemannian manifolds with boundary which are conformally St\"ackel, that is we show that the metric of such manifolds is uniquely determined by the Dirichlet-to-Neumann map measured on the boundary of the manifold, up to diffeomorphims that preserve the boundary.

\end{abstract}

%%% KEYWORDS %%%

\vspace{1cm}


%%% SUBJECTS CLASSIFICATION %%%%


%%% TABLE OF CONTENTS %%%

\tableofcontents

%%%%%%%%%%%%%%%%%%%%%%%%%%%%%%%%%%%%%%%%%%%%% INTRODUCTION %%%%%%%%%%%%%%%%%%%%%%%%%%%%%%%%%%%%%%%%%%%%%%%%%%%%%%%%%%%%%%%%%%%%%

\Section{Introduction}

\subsection{A prototype of an inverse problem: the anisotropic Calder\'on problem}

The anisotropic Calder\'on problem, named after the seminal paper \cite{Cal1980} by Alberto Pedro Calder\'on, addresses the question of determining the anisotropic conductivity of a body (\textit{i.e.} a domain in $\R^n$) from current and voltage measurements made only \emph{on the boundary} of the body, up to a change of coordinates fixing the boundary. It is well known (see \cite{LeU1989}) that, in dimension three or higher, this problem can be reformulated in a geometric manner as the problem of determining the Riemannian metric of a compact connected Riemannian manifold with boundary from the Dirichlet-to-Neumann (DN) map (the quantities measured on the boundary of the manifold), up to diffeomorphisms fixing the boundary \footnote{Note that in dimension $2$, there exists another gauge invariance for this problem due to the covariance of the Laplace-Beltrami operator under conformal changes of the metric. The anisotropic Calder\'on problem amounts then to determining the Riemannian metric of a smooth compact connected Riemannian manifold with boundary from the DN map measured on the boundary of the manifold, up to diffeomorphisms fixing the boundary and/or a conformal change of the metric}. 

There has been an intense activity around the anisotropic Calder\'on problem in the last 30 years. In dimension $2$, a complete positive answer was given in \cite{LeU1989, LU2001} in the case of smooth metrics and in \cite{ALP2005} for $L^\infty$ metrics. In dimension $3$ or higher, the anisotropic Calder\'on problem has been solved positively for \emph{real-analytic metrics} in the serie of papers \cite{LeU1989, LU2001, LTU2003}. All these works use crucially the boundary determination results of \cite{LeU1989} (see also \cite{KY2002} for a local version), that is the fact that the DN map determines uniquely the metric and all its derivatives (including the normal ones) on the boundary of the manifold. Then, the real-analyticity of the metric allows one to extend the boundary determination of the metric to the whole manifold, up to natural gauge invariances. We also refer to the recent paper \cite{LLS2018} where a new proof of uniqueness in the Calder\'on problem for real-analytic metrics is given that uses a different approach. The anisotropic Calder\'on problem was also solved for Einstein metrics (which are real-analytic in the interior of the manifold) in \cite{GSB2009}. 

The anisotropic Calder\'on problem for \emph{smooth} metrics in dimension $n\geq 3$ remains however a major open problem even though some important results have been obtained in the last decade in \cite{DSFKSU2009, DSFKLS2016} for classes of smooth compact connected Riemannian manifolds with boundary that are \emph{conformally transversally anisotropic} (CTA), meaning that
$$
  M \subset \subset \R \times M_0, \quad G = c ( e \oplus g_0),
$$
where $(M_0,g_0)$ is a $n-1$ dimensional smooth compact connected Riemannian manifold with boundary, $e$ is the Euclidean metric on the real line and $c$ is a smooth strictly positive function in the cylinder $\R \times M_0$. It has been shown in \cite{DSFKSU2009}, Theorem 1.2 and Lemma 2.9, that CTA manifolds are characterized by the existence of a limiting Carleman weight \footnote{In fact, the precise result states that if an open manifold $(M,G)$ admits a limiting Carleman weight, then it must be locally a CTA manifold.} or equivalently, by the existence of a nontrivial conformal Killing vector field. 
%\footnote{A vector field $X$ in $(M, G)$ is called a conformal Killing field if $L_X \, G = \lambda G$, where $L_X$ is the Lie derivative along $X$ and $\lambda$ is a scalar.}. 
The existence of a limiting Carleman weight can then be used to construct complex geometrical optics (CGO) solutions on the manifolds. Under some additional conditions on the transversal part $(M_0,g_0)$ of the cylinder (simplicity in \cite{DSFKSU2009} and injectivity of the geodesic ray-transform in \cite{DSFKLS2016}), CGO's technique can then be used to prove that the conformal factor $c$ is uniquely determined from the knowledge of the DN map. We refer to \cite{GT2013, KS2014, Sa2013, Uh2009, Uh2014} for surveys on the use of CGO's technique in the anisotropic Calder\'on problem. Finally, we mention the recent paper \cite{DSFKLLS2017} that provides a positive solution to the linearized Calder\'on problem on CTA manifolds under less restrictive conditions on the transversal part $(M_0,g_0)$.

Even though uniqueness is expected in the \emph{global} Calder\'on problem on \emph{smooth compact} connected Riemannian manifolds, some counterexamples to uniqueness are known in several cases where the above hypotheses fail. For instance, it is shown in \cite{LTU2003} that there exists a pair of compact and complete non-compact $2$-dimensional manifolds with boundary having the same DN map. This counterexample was obtained using a blow-up map. Some analogous non-uniqueness results for highly singular metrics on a compact manifold have been obtained in the study of invisibility phenomena in \cite{GLU2003, GKLU2007, GKLU2009a, GKLU2009b, DLU2017, DKN2018b}. The idea behind these cloaking devices is to hide an arbitrary object from measurements  by coating it with a meta-material that corresponds to a \emph{degenerate} Riemannian metric. Counterexamples to uniqueness in the case of \emph{H\"older continuous} metrics for a \emph{local} (meaning that the DN map is measured on a proper open subset of the boundary) Calder\'on problem has been obtained recently in \cite{DKN2019c} while several types of non-uniqueness results have been obtained in the Calder\'on problem with \emph{disjoint Dirichlet and Neumann data} in \cite{DKN2019a, DKN2019b, DKN2018a}. 

As conveyed by the title of this section, the anisotropic Calder\'on problem is the prototype of an inverse problem that shares many similarities with other important inverse problems; for instance, inverse scattering experiments on some non-compact Riemannian manifolds having ends where the question amounts to determining the metric from the knowledge of the scattering matrix at a fixed energy, the scattering matrix being understood as a kind of DN map measured at infinity (\textit{i.e.} the ends). Particularly close to the anisotropic Calder\'on problem on a compact manifold with boundary are the inverse scattering problems at fixed energy on \emph{asymptotically hyperbolic manifolds}, a class of manifolds in which the usual boundary is replaced by a conformal boundary of hyperbolic type. We refer to \cite{Iso2004a, Iso2004b, Iso2007a, Iso2007b} and especially the book \cite{IK2014}, Section 5.2. for the link between inverse boundary value problems and inverse scattering on asymptotically hyperbolic manifolds. Though a complete answer of the inverse scattering problem at a fixed energy is not known in that setting, some interesting and important results in that direction have been proved in \cite{JSB2000, SB2005, IK2014, IKL2017, HSB2015, HSB2018} for some general asymptotically hyperbolic manifolds. However, a complete positive answer was found in \cite{DN2011, DN2017, DGN2016, DKN2015, Go2018} for some Riemannian asymptotically hyperbolic manifolds or de-Sitter like black holes spacetimes having a sufficient amount of (hidden) symmetries. \\

In this paper, we wish to continue the series of works \cite{DN2011, DN2017, DGN2016, DKN2015, Go2018} by studying the anisotropic Calder\'on problem on \emph{conformally St\"ackel manifolds}, a class of $n$-dimensional completely integrable Riemannian manifolds with the property that the Laplace-Beltrami operator possess $n-1$ commuting second order conformal symmetry operators that allows to solve the corresponding Laplace PDE by separation of variables. Important to say in this introduction are the facts that :

\begin{itemize}
\item This class of $n$-dimensional conformally St\"ackel manifolds is rather large since we will show below that it depends locally on $n^2$ functions of one variable and a function of $n-1$ variables, precisely a function $\eta \in C^\infty(\partial M)$.

\item Moreover, it does not belong to the class of CTA manifolds that are characterized by the existence of a conformal Killing vector field. Instead, the class of conformally St\"ackel manifolds are "almost" characterized by the existence of $n-1$ independent \emph{conformal Killing tensors of rank 2} \footnote{A symmetric contravariant two-tensor $K=(K^{ij})$ is a conformal Killing tensor for the contravariant metric $G=(G^{ij})$ if there exists a vector field $X=(X^i)$ such that $[G,K] = 2 X \odot G$ where we denote by $[.,.]$ the Lie-Schouten bracket on contravariant symmetric tensors and by $\odot$ the symmetric tensor product.}. 

\item Since conformally St\"ackel manifolds are not CTA manifolds, we stress the fact that our proof of uniqueness in the Calder\'on problem does not use the usual CGO techniques but relies rather a mix of boundary determination results and what we would like to call the multi-parameter complex angular momentum (CAM) method, which is made possible thanks to the (hidden) symmetries of the manifolds under consideration. 

\end{itemize}

In order to keep the presentation of the model and of the ideas of the proof transparent, and to keep the notation as light as possible, we will now introduce conformally St\"ackel manifolds and solve the corresponding anisotropic Calder\'on problems in dimension three. Except for additional notational complexity, the extension to higher dimensions is in every regard identical.

%%% CONFORMALLY STAECKEL MANIFOLDS %%%%%%%%%%%%

\subsection{The model of conformally St\"ackel manifolds} 

We follow the presentation of conformally St\"ackel manifolds given in \cite{CR2006} (see also \cite{BCR2002a, BCR2002b, BCR2005}) and also refer to \cite{Ei1934, KM1980, KM1982, KM1983, KM1984, Mil1988, Rob1928, Sta1893} for other classical references in the variable separation theory. Even though the description of these models is only local in nature, we make it global by considering $\Omega$ to be a smooth compact connected three-dimensional manifold with smooth boundary having the global topology of a toric cylinder,
$$
 \Omega =[0,A] \times {\mathbb{T}}^{2}\,.
$$ 
Let us denote by $(x^1, x^2, x^3)$ a global coordinate system on $\Omega$ and note that the boundary $\partial \Omega$ of $\Omega$ has two connected components given by
$$
\partial \Omega = \Omega_{0}\cup \Omega_{1}\,,\quad \Omega_{0}=\{0\}\times  {\mathbb{T}}^{2}\,,\quad \Omega_{1}=\{A\}\times  {\mathbb{T}}^{2}\,.
$$
We equip the manifold $\Omega$ with a smooth Riemannian metric $G$ of the form  
\begin{equation} \label{MetricConf}
  G = c^{4} g = \sum_{i=1}^{3} H_{i}^{2}(dx^{i})^{2}\,.
\end{equation}
In the above expression, the Riemannian metric $g$ is a St\"ackel metric, that is  
\begin{equation} \label{MetricStackel}
  g = \sum_{i=1}^{3}{h}_{i}^{2}(dx^{i})^{2}\,, \quad h_{i}^{2}=\frac{\det S}{s^{i1}}\,,
\end{equation}
with $S$ being a St\"ackel matrix, that is a non-singular matrix of the form 
\begin{equation}\label{StackelMatrix}
S=\left(
\begin{matrix}
s_{11}(x^{1})&s_{12}(x^{1})&s_{13}(x^{1}) \\
s_{21}(x^{2})&s_{22}(x^{2}) & s_{23}(x^{2}) \\
s_{31}(x^{3})&s_{32}(x^{3})&s_{33}(x^{3})
\end{matrix}
\right)\,,
\end{equation}
and $s^{ij}$ denotes the cofactor of the component $s_{ij}$ of the matrix $S$. Observe that the diagonal components $h_{i}^{2}$ of the St\"ackel metric are given by the entries of the first column of the inverse St\"ackel matrix $A=S^{-1}$. Of course we demand the diagonal coefficients $h_i^2$ of the metric $g$ to be positive. 

Furthermore, the conformal factor $c^{4}$ is assumed to be a positive solution of the linear elliptic PDE on $\Omega$ given by:
\begin{equation} \label{PDEc}
-\Delta_{g}c - \sum_{i=1}^3 h_{i}^{2}\big(\phi_{i}+\frac{1}{4}\gamma_{i}^{2}-\frac{1}{2}\partial_{i}\gamma_{i}\big)c = 0\,.%+ \lambda^{2}c^{5}=0\,.
\end{equation}
where $\Delta_g$ denotes the Laplace-Beltrami operator associated to the St\"ackel metric $g$ on $\Omega$, given in local coordinates by
$$
  -\Delta_g = -\frac{1}{\sqrt{|g|}} \ \sum_{i,j=1}^{3}\  \partial_i \left( \sqrt{|g|} g^{ij} \partial_j \right)\,,
$$
and
$$
\gamma_{i}:=-\partial_{i}\log \frac{h_{1}h_{2}h_{3}}{h_{i}^{2}}\,,
$$
are the contracted Christoffel symbols associated to $g$ and $\phi_{i}=\phi_{i}(x^{i})$ are arbitrary smooth functions of the indicated variable. 

We shall review in Section \ref{II1} the theory of variable separation on conformally St\"ackel manifolds for the Hamilton-Jacobi and Laplace equations, and recall some intrinsic characterizations of such manifolds in terms of the existence of conformal Killing tensors having certain properties. However, let us emphasize here the remarkable fact that all solutions of the Laplace equation 
\begin{equation} \label{Laplace}
  -\Delta_G \, \psi = 0, \quad \textrm{on} \ \Omega,
\end{equation}
can be written as an infinite (countable) linear superposition of functions of the form 
$$
  \psi = R(x^1,x^2,x^3) u, \quad u = u_1(x^1) u_2(x^2) u_3(x^3),
$$
for a well chosen factor $R$. This will be crucial in the later analysis. More precisely, we will show that any solution $\psi$ of (\ref{Laplace}) can be written as
\begin{equation} \label{Decomposition}
  \psi = R(x^1,x^2,x^3) \sum_{m=1}^\infty u_m(x^1) Y_m(x^2,x^3), \quad Y_m(x^2,x^3) = v_m(x^2) w_m(x^3), 
\end{equation}
such that, for a convenient choice of factor $R$, each $u_m, v_m, w_m$ satisfies the coupled separated ODEs :
\begin{equation} \label{Separateda}
	-u_m'' + [\,\mu_m^2 s_{12}(x^1) + \nu_m^2 s_{13}(x^1) - \phi_1(x^1)] \, u_m = 0, 
\end{equation}
\begin{equation} \label{Separatedb}
	-v_m'' + [\,\mu_m^2 s_{22}(x^2) + \nu_m^2 s_{23}(x^2) - \phi_2(x^2)] \, v_m = 0, 
\end{equation}
\begin{equation} \label{Separatedc}	
	-w_m'' + [\,\mu_m^2 s_{32}(x^3) + \nu_m^2 s_{33}(x^3) - \phi_3(x^3)] \, w_m = 0. 
\end{equation}
Here the constants of separation $(\mu_m^2, \nu_m^2)$ can be understood as the joint spectrum of the commuting elliptic selfadjoint operators $(H,L)$ on $\mathbb{T}^2$ defined by :
\begin{equation} \label{HL}
  \left( \begin{array}{c} H \\ L \end{array} \right) = \frac{1}{s^{11}} \left( \begin{array}{cc} -s_{33} & s_{23} \\ s_{32} & -s_{22} \end{array} \right) \left( \begin{array}{c} A_2 \\ A_3 \end{array} \right),
\end{equation} 
where for all $j=1,2,3$, we set :
\begin{equation} \label{A}
  A_j = -\partial_{j}^2 - \phi_j(x^j). 
\end{equation}
The common eigenfunctions of $(H,L)$ take the form $Y_{m} = v_m(x^2) w_m(x^3)$ and satisfy (by definition) :
\begin{equation} \label{EigenHL}
  H Y_m = \mu_m^2 Y_m, \quad L Y_m = \nu_m^2 Y_m, \quad \forall m \geq 1. 
\end{equation}
Finally, the eigenfunctions $Y_m$ form a Hilbert basis of $L^2(\mathbb{T}^2)$ in the following sense :
\begin{equation} \label{HB}
  L^2(\mathbb{T}^2 \,; s^{11} dx^2 dx^3) = \bigoplus_{m \geq 1} \langle Y_m \rangle. 
\end{equation}
The proof of the above statement will be given at the beginning of Section \ref{II2}. As a consequence, we will be able to show that the DN map possesses a very special structure. Precisely, we will show that the DN map can be "almost" diagonalized onto the Hilbert basis $(Y_m)_{m \geq 1}$ as follows. First recall that the boundary of the cylinder $\Omega$ has two connected components $\Omega_0$ and $\Omega_1$ both isomorphic to $\mathbb{T}^2$. We thus identify the Sobolev spaces $H^s(\partial \Omega), \ s \in \R$, with
$$
  H^s(\partial \Omega) = H^s(\Omega_0) \oplus H^s(\Omega_1), \quad H^s(\Omega_j) \simeq H^s(\mathbb{T}^2), \ j=0,1, 
$$
and use a $2 \times 2$-matrix notation for the DN map $\Lambda_G\,: \, H^{\frac{1}{2}}(\partial \Omega) \longrightarrow H^{-\frac{1}{2}}(\partial \Omega)$, \textit{i.e.}
$$
  \Lambda_G = \left( \begin{array}{cc} \Lambda_{G,\Omega_0, \Omega_0} & \Lambda_{G,\Omega_0, \Omega_1} \\ \Lambda_{G,\Omega_1, \Omega_0} & \Lambda_{G,\Omega_1, \Omega_1} \end{array} \right). 
$$ 
Here the operators $\Lambda_{G,\Omega_i, \Omega_j}\, : \, H^{\frac{1}{2}}(\mathbb{T}^2) \longrightarrow H^{-\frac{1}{2}}(\mathbb{T}^2)$ correspond to the DN map when the Dirichlet data are imposed on $\Omega_i$ and the Neumann data are measured on $\Omega_j$. From (\ref{HB}), we see that for $s \geq 0$, any element of the Sobolev spaces $H^s(\Omega_j), \ j=0,1$, can be decomposed onto the Hilbert basis $(Y_m)_{m \geq 1}$. With these notations, we will show that the DN map on a conformally St\"ackel cylinder has the following structure : 
\begin{equation} \label{DNStructure}
  \Lambda_G = \left( \begin{array}{cc} \frac{-1}{H_1(0,x^2,x^3)} & 0 \\ 0 & \frac{1}{H_1(A,x^2,x^3)} \end{array} \right) \Bigg[ 
	            \left( \begin{array}{cc} \frac{\Gamma_1(0,x^2,x^3)}{2} & 0 \\ 0 & \frac{\Gamma_1(A,x^2,x^3)}{2} \end{array} \right) \hspace{2cm}
\end{equation}
$$							
	\hspace{2cm}	+ \ \left( \begin{array}{cc} R(0,x^2,x^3) & 0 \\ 0 & R(A,x^2,x^3) \end{array} \right) A_G \left( \begin{array}{cc} \frac{1}{R(0,x^2,x^3)} & 0 \\ 0 & \frac{1}{R(A,x^2,x^3)} \end{array} \right) \Bigg]
$$
where 
$$
\Gamma_{i}:=-\partial_{i}\log \frac{H_{1}H_{2}H_{3}}{H_{i}^{2}}\,, \ \ i=1,2,3, 
$$
are the contracted Christoffel symbols associated to the conformally St\"ackel metric $G$ and where the operator $A_G$ is completely diagonalizable onto the Hilbert basis $(Y_m)_{m \geq 1}$, its restriction on $\langle Y_m \rangle$ being defined by : 
\begin{equation} \label{AG}
  (A_G)_{|\langle Y_m \rangle} := \left( \begin{array}{cc} M(\mu_m^2,\nu_m^2)  & \frac{1}{\Delta(\mu_m^2,\nu_m^2)} \\ \frac{1}{\Delta(\mu_m^2,\nu_m^2)} & N(\mu_m^2,\nu_m^2) \end{array} \right).
\end{equation}
Finally, the function $\Delta(\mu^2,\nu^2)$ and the functions $M(\mu^2,\nu^2)$ and $N(\mu^2,\nu^2)$ are the (not so classical) \emph{characteristic} and \emph{Weyl-Titchmarsh} (WT) functions associated to the radial ODE (\ref{Separateda}) with Dirichlet boundary conditions. We emphasize the worlds \emph{not so classical} since the characteristic and WT functions depend on the \emph{two} spectral parameters $\mu^2,\nu^2$ which appear as the constants of separation in the variable separation procedure. 

The construction and the explanation of this special structure of the DN map as well as a review of the elementary properties of the characteristic and WT functions will be done in Section \ref{II2}. 

For the moment, since we are interested in the anisotropic Calder\'on problem on conformally St\"ackel manifolds, let us simply (and formally) count the number of unknown functions defining them. A priori, a St\"ackel metric $g$ depends on nine functions $s_{ij}(x^i)$ of one variable while the conformal factor $c$ depends on three additional unknown functions $\phi(x^i)$ through the Laplace type pde (\ref{PDEc}). Let us choose (this is always possible) the functions $\phi_i$ in such a way that the zero-order term be nonnegative, \textit{i.e.}
\begin{equation} \label{H1c}
- \sum_{i=1}^3 h_{i}^{2}\big(\phi_{i}+\frac{1}{4}\Gamma_{i}^{2}-\frac{1}{2}\partial_{i}\Gamma_{i}\big) \geq 0,
\end{equation}
and solve the Dirichlet problem
\begin{equation} \label{DirichletPbc}
  \left\{ \begin{array}{rcl} 
	-\Delta_{g}c - \sum_{i=1}^3 h_{i}^{2}\big(\phi_{i}+\frac{1}{4}\Gamma_{i}^{2}-\frac{1}{2}\partial_{i}\Gamma_{i}\big)c & = & 0 \quad \textrm{on} \ \ \Omega, \\
	c & = & \eta, \quad \textrm{on} \ \ \partial \Omega. 
	\end{array} \right. 
\end{equation}
According to the maximum principle \cite{Ev2010, Tay2011a}, for any positive boundary function $\eta$ on $\partial \Omega$, there exists a unique positive solution $c$ of (\ref{DirichletPbc}). We conclude that the conformal factor $c$ depends roughly speaking on three unknown functions $\phi_i(x^i)$ of one variable (that satisfiy (\ref{H1c})) and a positive function $\eta \in C^\infty(\partial \Omega)$.  

This makes twelve unknown functions of one variable and one unknown function of two variables for the metric $g$. Nevertheless, it is possible to remove three of the unknown functions of one variable by a simple change of coordinates that preserves the conformally St\"ackel structure. Indeed, given positive functions $f_i(x^i)$, define the new variables $y^i$ by : 
\begin{equation} \label{CV}
  y^i = \int_0^{x^i} \sqrt{f_i(s)} ds. 
\end{equation}
Then the new metric $\bar{G}$ is given 
\begin{equation} \label{NewG}
  \bar{G} = \bar{c}^4 \bar{g},
\end{equation}
where $\bar{g}$ is the metric 
\begin{equation} \label{Newg}
  \bar{g} = \sum_{i=1}^{3} \bar{h}_{i}^{2}(dy^{i})^{2}\,, \quad \bar{h}_{i}^{2} = \frac{h_i^2(x^i(y^i))}{f_i(x^i(y^i))},
\end{equation}
which can be shown to be a St\"ackel metric
\begin{equation} \label{Newgs}
  \bar{g} = \sum_{i=1}^{3} \bar{h}_{i}^{2}(dy^{i})^{2}\,, \quad \bar{h}_{i}^{2} = \frac{\det \bar{S}}{\bar{s}^{i1}}\,,
\end{equation}
associated to the new St\"ackel matrix
\begin{equation} \label{NewS}
  \bar{S} = \left( \bar{s}_{ij}(y^i) \right)_{1 \leq i,j \leq 3} := \left( \frac{s_{ij}(x^i(y^i))}{f_i(x^i(y^i))} \right)_{1 \leq i,j \leq 3}.
\end{equation}
In other words, the change of variables (\ref{CV}) amounts to dividing each line of the initial St\"ackel matrix by the functions $f_i$, a step which allows us to remove one unknown function in each variable $x^i$. Note finally that the conformal factor $\bar{c}$ now satisfies : 
$$
-\Delta_{\bar{g}} \bar{c} - \sum_{i=1}^3 \bar{h}_{i}^{2}\big( \bar{\phi}_{i} + \frac{1}{4}\bar{\gamma}_{i}^{2} - \frac{1}{2}\partial_{i}\bar{\gamma}_{i} \big) \bar{c} = 0\,,
$$
where the arbitrary functions $\bar{\phi_i} = \bar{\phi_i}(y^i)$ are given by
\begin{equation} \label{NewPhii}
  \bar{\phi}_i = \frac{\phi_i}{f_i} - \frac{(\dot{\log f_i})^2}{16} - \frac{(\ddot{\log f_i})}{4},
\end{equation}
(here the dot denotes the derivative with respect to $y^i$) and the $\bar{\gamma}_i$ are the contracted Christoffel symbols associated to the metric $\bar{g}$. In conclusion, we deduce that a conformally St\"ackel metric effectively depends on $9 = 3^2$ unknown functions of one variable and one positive function $\eta \in C^\infty(\partial \Omega)$ of $2$ variables.

%%% MAIN RESULTS

\subsection{The results} 

We will study the anisotropic Calder\'on problem in the class of smooth compact connected Riemannian manifolds with boundary $(M,G)$ that are embedded in a conformally St\"ackel cylinder $\Omega$, \textit{i.e.} we will consider $(M,G)$ where
\begin{equation} \label{Embed}
  M \subset \subset \Omega = [0,A] \times \mathbb{T}^2,
\end{equation}
and $G$ is a Riemannian metric on $M$ that possesses a smooth extension (still denoted by $G$) to the whole cylinder $\Omega$ given by (\ref{MetricConf}) - (\ref{StackelMatrix}). 

Let us consider the corresponding Dirichlet problem
\begin{equation} \label{Eq000}
  \left\{ \begin{array}{cc} -\Delta_G\, u = 0, & \textrm{on} \ M, \\ u = f, & \textrm{on} \ \partial M. \end{array} \right.
\end{equation}
It is well known \cite{Sa2013} that, for any $f \in H^{1/2}(\partial M)$, there exists a unique weak solution $u \in H^1(M)$ of the Dirichlet problem (\ref{Eq000}). So, we can define the Dirichlet-to-Neumann map as the operator $\Lambda_{G}$ from $H^{1/2}(\partial M)$ to $H^{-1/2}(\partial M)$ given by
\begin{equation} \label{DN-Abstract}
  \Lambda_{G} (f) = \left( \partial_\nu u \right)_{|\partial M}.
\end{equation}
Here, $u$ is the unique solution of (\ref{Eq000}) and $\left( \partial_\nu u \right)_{|\partial M}$ is its normal derivative with respect to the unit outer normal $\nu$ on $\partial M$. Note that this normal derivative has to be understood in the weak sense as an element of $H^{-1/2}(\partial M)$ via the bilinear form
$$
  \left\langle \Lambda_{G} (f) \, , h \right \rangle = \int_M \langle du, dv \rangle_G \  \textrm{dVol}_G,
$$
where $f, h  \in H^{1/2}(\partial M)$, $u$ is the unique solution of the Dirichlet problem (\ref{Eq000}), and where $v$ is any element of $H^1(M)$ such that $v_{|\partial M} = h$. Of course, when $f$ is sufficiently smooth, this definition coincides with the usual one in local coordinates, that is
\begin{equation} \label{DN-Coord}
\partial_\nu u = \sum_i \nu^i \partial_i u.
\end{equation}
Finally, we will use the notations $\Lambda_G = \Lambda_{G,M}$ and $\Lambda_{G,\Omega}$ to distinguish between the DN map measured on $M$ and $\Omega$ respectively.

Let us now formulate our main result. 

\begin{thm} \label{Main}
  Let $(M,G)$ and $(M,\tilde{G})$ be two conformally St\"ackel manifolds satisfying (\ref{MetricConf}) - (\ref{StackelMatrix}) and (\ref{Embed}). We will add a subscript $\ \tilde{}$ to all the quantities related to $(M,\tilde{G})$. Assume that
$$
  \Lambda_G = \Lambda_{\tilde{G}}.
$$	
Then there exists a diffeomorphism $\varphi: \, M \longrightarrow \, M$ with $\varphi_{|\partial M} = Id$ whose pull-back satisfies
$$
  \tilde{G} = \varphi^* G, 
$$
\end{thm}
Let us make some comments on this result.

\begin{itemize}
\item [\textbf{1.}] The diffeomorphim $\varphi$ appearing in the statement of the main Theorem is simply a change of variables of the special form (\ref{CV}) that preserves the structure of conformally St\"ackel manifolds. 

\item [\textbf{2.}] Theorem \ref{Main} solves positively the uniqueness issue in the Calder\'on problem on conformally St\"ackel manifolds. It is an extension of the results in \cite{Go2018} where the inverse scattering problem at a fixed energy on St\"ackel asymptotically hyperbolic manifolds was considered. One of the main differences between the model in \cite{Go2018} and our model is the fact that in \cite{Go2018} the conformal factor $c$ is assumed to be identically $1$ and the PDE (\ref{PDEc}) on the conformal factor $c$ is then replaced by the so called Robertson conditions $\partial_j \gamma_i = 0, \ \forall 1 \leq i \ne j \leq 3$ (see \cite{Rob1928}). Note that under these hypotheses, the PDE (\ref{PDEc}) is trivially satisfied. The Robertson conditions restrict the class of St\"ackel metrics drastically and this is the sense in which our result extends \cite{Go2018}. We refer to \cite{Go2018}, Example 1.2, for a list of examples of St\"ackel metrics satisfying the Robertson conditions. 

\item [\textbf{3.}]  As already mentioned, conformally St\"ackel manifolds aren't generically CTA manifolds. It would be the case however if one of the line in the St\"ackel matrix $S$ was a line of constant functions. Assume for instance that the $s_{1j}$ are constants for all $j=1,2,3$. Then $\partial_{x^1}$ is a Killing vector field for the St\"ackel metric $g$ and thus a conformal Killing vector field for the metric $G$. Hence $(M,G)$ would lie within the class of CTA manifolds. This is clear if we notice that the metric $G$ can then be written as:
$$
  G = \left( c^4 \frac{\det(S)}{s^{11}} \right) \left[ (dx^1)^2 + g_0 \right], \quad g_0 = \frac{s^{11}}{s^{21}}(dx^2)^2 + \frac{s^{11}}{s^{31}}(dx^3)^2. 
$$
Even in that case however, we could not apply directly the results of \cite{DSFKSU2009, DSFKLS2016} since the injectivity of the  geodesic X-ray transform on the \emph{closed} transversal manifold 
$$
  (M_0,g_0) = (\mathbb{T}^2, g_0),
$$
is not guaranteed in general. 
\end{itemize}

% SKETCH OF THE PROOF

The proof of Theorem \ref{Main} will be divided in four steps. \\

\noindent \textit{Step 1. Extension to the whole cylinder $\Omega$.} Note first that the Laplace equation $-\Delta_G \,\psi = 0$ on $M$ is usually not separable since the boundary $\partial M$ need not be compatible with variable separation, unlike the case on the whole cylinder $(\Omega,G)$. Hence we cannot use a priori the form (\ref{Decomposition}) for the solutions of the Laplace equation as well as the structure (\ref{DNStructure}) of the DN map. However we can reduce the Calder\'on problem on $(M,G)$ to the Calder\'on problem on the extended cylinder $(\Omega,G)$ by the following result which is similar to the corresponding results on asymptotically hyperbolic manifolds from \cite{IK2014}, chapter 5, Theorems 2.3 and 4.6. 

\begin{thm} \label{ExtensionThm}
  Let $M_1 \subset \subset M_2$ be two smooth compact connected manifolds with boundary. Let $G$ and $\tilde{G}$ be two Riemannian metrics on $M_2$ such that $G = \tilde{G}$ on $M_2 \setminus M_1$. Denote by $\Lambda_{G,j}$ the DN map associated to $G$ on $M_j$ for $j=1,2$. Then 
$$
  \Lambda_{G,1} = \Lambda_{\tilde{G},1} \ \ \Longrightarrow \ \ \Lambda_{G,2} = \Lambda_{\tilde{G},2}. 	
$$	
\end{thm}
Together with the well-known boundary determination results from \cite{KY2002, LeU1989} or \cite{DSFKSU2009}, Section 8, we will deduce from Theorem \ref{ExtensionThm} :

\begin{prop} \label{ExtensionCalderon}
Assume that the hypotheses of Theorem \ref{ExtensionThm} hold. Then
$$
  \Lambda_{G,M} = \Lambda_{\tilde{G},M} \ \ \Longrightarrow \ \ \Lambda_{G,\Omega} = \Lambda_{\tilde{G},\Omega},	
$$
where the extended metrics $G$ and $\tilde{G}$ on the whole cylinder $\Omega$ are conformally St\"ackel metrics that can be chosen so as to satisfy $G = \tilde{G}$ on $\Omega \setminus M$ and the generic condition :
\begin{equation} \label{Generic}
  \left( \begin{array}{c} -s_{13}(0) \\ s_{12}(0) \end{array} \right), \  \left( \begin{array}{c} -s_{13}(A) \\ s_{12}(A) \end{array} \right) \ \textrm{are linearly independent}. 
\end{equation}
\end{prop}

In conclusion, it will be enough to prove uniqueness in the Calder\'on problem on conformally St\"ackel cylinders $(\Omega,G)$ for which we can use separation of variables. The proof of Theorem \ref{ExtensionThm} and Proposition \ref{ExtensionCalderon} will be given in Section \ref{III1}. \\

%Note that the same extension Theorem \ref{ExtensionThm} and standard boundary determination results allow to solve completely the anisotropic Calder\'on problem on CTA manifolds whose geodesic X-ray transforms on the transversal manifolds are injective by showing that both the conformal factor $c$ and the transversal metric $g_0$ are uniquely determined by the DN map. This will be done in Appendix \ref{A1}. \\

\noindent \textit{Step 2. Boundary determination.} After reducing the Calder\'on problem to the whole conformally St\"ackel cylinders $\Omega$ satisfying the conclusions of Proposition \ref{ExtensionCalderon}, we use the standard boundary determination results \footnote{Precisely we use the fact $\Lambda_{G,\Omega} = \Lambda_{\tilde{G},\Omega}$ imply the equality of $G_{|\partial \Omega}$ and $\tilde{G}_{| \partial \Omega}$ as well as the equality between the normal derivatives $(\partial_\nu G)_{|\partial \Omega}$ and $(\partial_{\tilde{\nu}} \tilde{G})_{|\partial \Omega}$ on the boundary $\partial \Omega$.} from \cite{KY2002, LeU1989} and the particular structure of the metrics $G$ and $\tilde{G}$ given by (\ref{MetricConf}) - (\ref{StackelMatrix}) to prove in a successive series of steps that first (from the equality of the metrics on the boundary)
\begin{equation} \label{T=}
  \left( \begin{array}{cc} s_{22} & s_{23} \\ s_{32} & s_{33} \end{array} \right) = \left( \begin{array}{cc} \tilde{s}_{22} & \tilde{s}_{23} \\ \tilde{s}_{32} & \tilde{s}_{33} \end{array} \right), 
\end{equation}
as functions of $x^2,x^3$ and
\begin{equation} \label{G=}
  \left\{ \begin{array}{c}
	(c^4 \, \det S) (x^1,x^2,x^3) = (\tilde{c}^4 \, \det \tilde{S})(x^1,x^2,x^3), \\  
  R(x^1,x^2,x^3) = \tilde{R}(x^1,x^2,x^3), \\  
  H_1(x^1,x^2,x^3) = \tilde{H}_1(x^1,x^2,x^3), 
	\end{array} \right. \quad x^1 = 0, A, \ \forall x^2, x^3.
\end{equation}
and second (from the equality of the normal derivatives of the metrics on the boundary)
\begin{equation} \label{Gnu=}
\left\{ \begin{array}{c}
  (\partial_1 \log c^4 \, \det S) (x^1,x^2,x^3) = ( \partial_1 \log \tilde{c}^4 \, \det \tilde{S}) (x^1,x^2,x^3), \\
  \Gamma_1(x^1,x^2,x^3) = \tilde{\Gamma}_1(x^1,x^2,x^3), 
	\end{array} \right. \quad x^1 = 0, A, \ \forall x^2, x^3.
\end{equation}
where we recall that
$$
\Gamma_{i}:=-\partial_{i}\log \frac{H_{1}H_{2}H_{3}}{H_{i}^{2}}\,, \ \ i=1,2,3. 
$$
Then, using the special structure (\ref{DNStructure}) of the DN map, we will infer from (\ref{T=}) - (\ref{Gnu=}) that 
\begin{equation} \label{AG=}
  A_G = A_{\tilde{G}},
\end{equation}
where the operator $A_G$ is defined in (\ref{AG}). From this and some additional work, we will be able to show the equality of the eigenfunctions $Y_m$   
\begin{equation} \label{Ym=}
  Y_m = \tilde{Y}_m, \ \forall m,
\end{equation}
the equality of the joint spectra 
\begin{equation} \label{JointS=}
  (\mu_m^2, \nu_m^2 )= (\tilde{\mu}_m^2, \tilde{\nu}_m^2), \quad \forall m,
\end{equation}
and the equality between the $\phi_2$ and $\phi_3$  
\begin{equation} \label{Phi23=}
  \phi_2 = \tilde{\phi}_2, \quad \phi_3 = \tilde{\phi}_3.  
\end{equation}
Hence at the end of the second step, we will have recovered most of the unknown functions of one variable depending on one of the angular variables $x^2,x^3$, and in fact all of them if we keep in mind the possibility of removing some of these unknown functions thanks to the change of variables (\ref{CV}). 

Note finally that the above results clearly depend on the particular structure of conformally St\"ackel metrics on the cylinder $\Omega$ but also on a clear understanding of the different invariances of the St\"ackel metrics with respect to different but equivalent choices of the associated St\"ackel matrices. These invariances will be explained in Section \ref{II1} whereas the boundary determination results and their consequences will be given in Section \ref{III2}. \\ 

\noindent \textit{Step 3. The multi-parameter CAM method.} At this stage, it remains essentially to determine the unknown functions depending on the radial variable $x^1$ and the conformal factor $c$. To determine the former, we start from the equality
\begin{equation} \label{MJS=}
  M(\mu_m^2,\nu_m^2) =\tilde{ M}(\mu_m^2,\nu_m^2), \quad \forall m,
\end{equation}
which is a consequence of (\ref{AG=}) and (\ref{JointS=}). Recall that the WT function $M$ only depends on the radial ODE (\ref{Separateda}) and contains all the information on the functions $s_{12}, s_{13}, \phi_1$ through the well known Borg-Marchenko Theorem \cite{Be2001, Bo1946, Bo1952, GS2000}. Our first task is thus to extend the equality (\ref{MJS=}) which is initially true on the joint spectrum $J = \{(\mu_m^2,\nu_m^2), \ m \geq 1 \}$ to the whole plane $\C^2$, that is we aim to complexify the angular momenta as it was done for the first time by Regge in \cite{Re1959} and applied to solve some inverse problems in \cite{DeAR1965, DN2011, DN2016, DN2017, DGN2016, DKN2015, DKN2018a, DKN2018b, DKN2019a, DKN2019b, Go2018, Lo1968, New2002, Ra1999} and references therein. For this, we use a multi-parameter CAM method as in \cite{Go2018} which allows us to prove that
\begin{equation} \label{M=}
  M(\mu^2,\nu^2) =\tilde{ M}(\mu^2,\nu^2), \quad \forall \mu, \nu \in \C \backslash\{poles\}
\end{equation}
Once it is done, an application of Borg-Marchenko Theorem leads to
\begin{equation} \label{s1jphi1=}
  \phi_1 = \tilde{\phi}_1, \quad s_{12} = \tilde{s}_{12}, \quad s_{13} = \tilde{s}_{13},  
\end{equation}
up to a change of $x^1$-variable of the type (\ref{CV}). We would like to emphasize that the multi-parameter CAM method that permits to infer (\ref{M=}) from (\ref{MJS=}) is far from being as simple as in the case of a single angular momentum. Indeed, the method lies within the realm of functions of several complex variables and not one complex variable. Moreover, a good understanding of the joint spectrum $J$ is needed. We follow here the corresponding results obtained by Gobin \cite{Go2018}, which should be useful in other contexts as well. The results on the CAM method will be given in Section \ref{III3}. \\

\noindent \textit{Step 4. A unique continuation argument for the conformal factor.} We finish the proof of our main Theorem by remarking first that the metric $G$ can be written as
$$
  G = \alpha \, g_0, \quad \alpha = c^4 \det S, \quad g_0 = \frac{1}{s^{11}}(dx^1)^2 + \frac{1}{s^{21}}(dx^2)^2 + \frac{1}{s^{31}}(dx^3)^2, 
$$
Note from the results of Steps 1 to 3 that we have 
\begin{equation} \label{g0}
  g_0 = \tilde{g_0},
\end{equation}	
up to a change of coordinates (\ref{CV}). Thus it only remains to prove that $\alpha = \tilde{\alpha}$. The second crucial remark consists in using (\ref{PDEc}) to show that the conformal factor $\alpha$ satisfies the elliptic PDE
\begin{equation} \label{PDEalpha}
  -\Delta_{g_0} \alpha - Q_{g_0,\phi_i} \alpha = 0, 
\end{equation} 
where
\begin{equation} \label{Q}
  Q_{g_0,\phi_i}=\sum_{i=1}^3 g_0^{ii} \left[ \frac{\partial^2_{ii} \log \det g_0}{4}  + \frac{\partial_{i} \log \det g_0}{8} + \frac{( \partial_{i} \log \det g_0)^2}{16} + \phi_i \right].  
\end{equation}
Thanks to (\ref{Phi23=}), (\ref{s1jphi1=}) and (\ref{g0}), we thus observe one additional (and last) remarkable fact: the conformal factors $\alpha$ and $\tilde{\alpha}$ satisfy the \emph{same} second order elliptic PDE (\ref{PDEalpha}). Finally, we use (\ref{G=}), (\ref{Gnu=}) and a classical unique continuation principle \cite{HoI2013, Tat1995, Tat2004} to prove $\alpha = \tilde{\alpha}$. As a consequence, we find that
$$
  G = \tilde{G},
$$
up to some isometries of the type (\ref{CV}) that preserve the boundary. The derivation of the elliptic PDE (\ref{PDEalpha}) satisfied by $\alpha$ and the unique continuation argument will be given in Section \ref{III4}.

%%%%%%%%%%%%%%%%%%%%%%%%%%%%%%%%%%%%%%%%%%%%%% THE DN MAP %%%%%%%%%%%%%%%%%%%%%%%%%%%%%%%%%%%%%%%%%%%%%%%%%%%%%%%%%%%%%%%%

\Section{The DN map on conformally St\"ackel manifolds}

\subsection{A review of complete integrability and separability's properties on conformally St\"ackel manifolds} \label{II1}

% STACKEL MANIFOLDS

\noindent \textit{The case of St\"ackel manifolds}. St\"ackel manifolds (or systems) date back to the work by St\"ackel \cite{Sta1893}, Robertson \cite{Rob1928} and Eisenhart \cite{Ei1934} on the theory of orthogonal variable separation for the Hamilton-Jacobi (HJ) equation $g(\nabla u, \nabla u) = E$ and the Helmholtz equation $-\Delta_g \, \psi = E \psi$ on a $n$-dimensional pseudo-Riemannian manifold $(M,g)$. Here by \emph{orthogonal separation}, we mean that we look for diagonal metrics $g$ satisfying $g_{ij} = 0, \ i \ne j$ such that : 
\begin{itemize}
\item the HJ equation possesses locally a solution $u(x,c)$ parametrized by $n$ constants $c = (c_1, \dots, c_n)$ of the form 
$$
  u(x,c) = \sum_{i=1}^n u_i(x^i,c), \quad x = (x^1, \dots, x^n), 
$$
satisfying the completeness condition
$$
  \det \left[ \frac{ \partial^2 u}{\partial x^i \partial c_j} \right] \ne 0.
$$
\item the Helmholtz equation possesses locally a solution $\psi(x,c)$ parametrized by $2n$ constants $c = (c_1, \dots, c_{2n})$ of the form 
$$
  \psi(x,c) = \prod_{i=1}^n \psi_i(x^i,c), \quad x = (x^1, \dots, x^n), 
$$
satisfying the completeness condition
$$
  \det \left[ \begin{array}{c} \frac{ \partial u_i}{\partial c_J} \\ \frac{ \partial v_i}{\partial c_J} \end{array} \right] \ne 0, \quad u_i = \frac{\psi_i'}{\psi_i}, \quad v_i = \frac{\psi_i''}{\psi_i}.
$$
\end{itemize}
The three classical theorems of St\"ackel, Robertson and Eisenhart are as follows 

\begin{thm}[St\"ackel, 1893] \label{Stackel}
The HJ equation is separable in orthogonal coordinates $x=(x^1, \dots, x^n)$ for all values of the energy $E$ if and only if the metric is of the form 
$$
  g = \sum_{i=1}^{n}{h}_{i}^{2}(dx^{i})^{2}\,, \quad h_{i}^{2}=\frac{\det S}{s^{i1}}\,,
$$
with $S=\left( s_{ij}(x^i) \right)$ being a St\"ackel matrix, that is a non-singular matrix such that each entry $s_{ij}$ depends only the coordinate $x^i$, and $s^{ij}$ denotes the cofactor of the component $s_{ij}$ of the matrix $S$. 
\end{thm}

\begin{thm}[Robertson, 1928] \label{Robertson}
The Helmholtz equation is separable in orthogonal coordinates $x=(x^i)$ for all values of the energy $E$ if and only if in these coordinates the metric $g$ is St\"ackel and moreover the following Robertson condition is satisfied
$$
  \partial_j \gamma_i = 0, \quad i \ne j, 
$$
with 
$$
  \gamma_{i}:=-\partial_{i}\log \frac{h_{1}h_{2}h_{3}}{h_{i}^{2}}\,.
$$
\end{thm}

\begin{thm}[Eisenhart, 1934] \label{Eisenhart}
  The Robertson condition is satisfied if and only if the Ricci tensor is diagonal, \textit{i.e.} $R_{ij} = 0$ for $i \ne j$. 
\end{thm}

More intrinsic characterizations of the separability properties of the HJ and Helmholtz equations have been obtained later by Kalnins and Miller (see for instance \cite{KM1980} and the survey \cite{Mil1988}) and Benenti (see for instance the surveys \cite{BCR2002a, BCR2002b}). In order to state them, let us recall some standard definitions. Let $K=(K^{ij})$ be a symmetric contravariant two-tensor. We denote by $P_K$ the fiber-wise homogeneous polynomial function on the cotangent bundle $T^*M$ given by $P_K = K^{ij}p_i p_j$. We say that two such symmetric contravariant two-tensors $K$ and $K'$ are in involution if the corresponding polynomial functions are in involution, \textit{i.e.} if their Poisson bracket relative to the canonical symplectic structure of $T^{*}M$ vanishes identically
$$
  \{ P_K, P_{K'} \} = 0. 
$$
We say that a symmetric contravariant two-tensor $K$ is a Killing tensor on $(M,g)$ if and only if $P_K$ is a first integral of the geodesic flow, \textit{i.e.} it is in involution with the geodesic Hamiltonian $H = g^{ij}p_i p_j$, \textit{i.e.} 
\begin{equation}\label{Poissbrack}
\{ P_K, H \} = 0, 
\end{equation}
or equivalently
$$
  \nabla^{(h} K^{ij)} = 0, 
$$ 
where $\nabla$ denotes the covariant derivative with respect to the Levi-Civita connection and the parentheses $(\dots)$ denotes the symmetrization of the indices. Finally, to a symmetric contravariant two-tensor $K = (K^{ij})$, we can associate by means of the pseudo-Riemannian metric $g$ a linear operator $\mathbf{K}$ acting on vector fields $X=(X^i)$ by means of 
$$
  (\mathbf{K} X )^i = K^{ij}g_{jh}X^h = K^i_{\, j} X^j.
$$
Hence we can talk about eigenvalues, eigenvectors, etc... of a symmetric contravariant two-tensor $K$ throught this identification. 

Let us state now the intrinsic characterizations of St\"ackel metrics in the formulation given in \cite{BCR2002a}. 

\begin{thm}[Intrinsic characterizations for HJ and Helmholtz equations] \label{HJH} 
1) The HJ equation on $(M,g)$ is orthogonally separable if and only if there exists a Killing tensor $K$ with pointwise simple eigenvalues and normal (\textit{i.e.} orthogonally integrable or surface forming) eigenvectors. \\
2) The HJ equation on $(M,g)$ is orthogonally separable if and only if there exist $n$ pointwise independent Killing tensors $K_a, \ a=1,\dots,n$ commuting as linear operators and in involution. Moreover, the contravariant metric $g=g^{ij}$ belongs to the algebra generated by the tensors $K_a$ and can be chosen equal to $K_1$. \\
3) The Helmholtz equation is orthogonally separable on $(M,g)$ if and only of there exists a Killing tensor $K$ with simple eigenvalues and normal eigenvectors that commutes with the Ricci tensor, \textit{i.e.} {\bf $K^{ij} R_{jk} - R_{ij} K^{jk} = 0$}. 
\end{thm}

Finally, we make the link between the above intrinsic characterization of the separability of the Helmholtz equation with the existence of second order symmetry operators for the Laplace-Beltrami operator $-\Delta_g$. First let us associate to the $n$ Killing tensors $K_a, \ a=1,\dots,n$ from Theorem \ref{HJH}, the pseudo-Laplacian $\Delta_{K_a}$ by
$$
  \Delta_{K_a} \psi = \nabla_i(K_a^{ij} \nabla_j) \psi,
$$
where $\nabla_i$ denotes the covariant derivative with respect to the Levi-Civita connection. Then, we have the following result (see Theorems 6.2 and 6.3 in \cite{BCR2002b})

\begin{thm} \label{SymmetryOp}
  All pseudo-Laplacian $\Delta_{K_a}, \ a=1,\dots,n$ pairwise commute and thus (since $\Delta_{K_1} = \Delta_g$) commute with the Laplace-Beltrami operator $\Delta_g$.  
\end{thm}

% CONFORMALLY STACKEL MANIFOLDS

\noindent \textit{The case of conformally St\"ackel manifolds}. The above separability results in the HJ and Helmholtz equations are valid \emph{at all energies $E$}. What happens if the energy $E$ is fixed and furthermore set equal to $0$? Note that in the Hamilton-Jacobi case, the corresponding orthogonal variable separation theory would only apply to the case of null geodesics and would therefore require that the metric have indefinite signature. Even though we shall eventually only be concerned with the case of Riemannian signature for the purposes of the Calder\'on problem studied in this paper, we shall for now recall the definitions and characterizations of separability for the null HJ and Laplace-Beltrami equations in general pseudo-Riemannian signature, following the classical results of Kalnins and Miller \cite{KM1982, KM1983, KM1984} and Benenti, Chanu and Rastelli \cite{BCR2005, CR2006}. In the Riemannian case, this will give rise to the class of conformally St\"ackel manifolds studied in this paper. 

First, the definitions of separated solutions of the null HJ and Laplace equations slightly differ from the previous ones since we now allow $R$-separability. Precisely 

\begin{itemize}
\item the null HJ equation is said to be separable if it possesses locally a solution $u(x,c)$ parametrized by $n-1$ constants $c = (c_1, \dots, c_{n-1})$ of the form 
$$
  u(x,c) = \sum_{i=1}^n u_i(x^i,c), \quad x = (x^1, \dots, x^n), 
$$
satisfying the rank condition
$$
  \textrm{rank} \left[ \frac{ \partial^2 u}{\partial x^i \partial c_j} \right] = n-1.
$$
\item the Laplace equation is said to be $R$-separable if there exists a function $R$ such that the Laplace equation possesses locally a solution $\psi(x,c)$ parametrized by $2n-1$ constants $c = (c_1, \dots, c_{2n-1})$ of the form 
$$
  \psi(x,c) = R \,\prod_{i=1}^n \psi_i(x^i,c), \quad x = (x^1, \dots, x^n), 
$$
satisfying the rank condition
$$
  \textrm{rank} \left[ \begin{array}{c} \frac{ \partial u_i}{\partial c_J} \\ \frac{ \partial v_i}{\partial c_J} \end{array} \right] = 2n-1, \quad u_i = \frac{\psi_i'}{\psi_i}, \quad v_i = \frac{\psi_i''}{\psi_i}.
$$ 
\end{itemize}

Second, we say that orthogonal coordinates $x = (x^i)$ are \emph{conformally separable} on a Riemannian manifold $(M,G)$ if there exists a smooth positive function $c$ (playing the role of a conformal factor) and a St\"ackel metric $g$ such that 
$$
  G = c^4 \, g = \sum_{i=1}^n H_i^2 (dx^i)^2.
$$
Then we have the following characterizations :

\begin{thm}[Kalnins-Miller \cite{KM1983}, Benenti-Chanu-Rastelli \cite{BCR2005}]
  The null HJ equation is separable in orthogonal coordinates $x = (x^i)$ if and only if these coordinates are conformally separable. 
\end{thm} 

\begin{thm}[Kalnins-Miller \cite{KM1984}, Chanu-Rastelli \cite{CR2006}]
  The Laplace equation is separable in orthogonal coordinates $x = (x^i)$ if and only if these coordinates are conformally separable and there exist functions $\phi_i = \phi_i(x^i), \ i=1,\dots,n$ such that the generalized Robertson condition is satisfied 
\begin{equation} \label{GenRob}
  \frac{1}{4} \sum_{i=1}^n G^{ii} ( 2 \partial_i \Gamma_i - \Gamma_i^2) = \sum_{i=1}^n G^{ii} \phi_i,
\end{equation}
with 
$$
  \Gamma_{i}:=-\partial_{i}\log \frac{H_{1}H_{2}H_{3}}{H_{i}^{2}}\,.
$$
In this case, the function $R$ is any solution of
\begin{equation} \label{RDef}
  2 \partial_i \ln R = \Gamma_i - \xi_i(x^i), 
\end{equation}
for arbitrary functions $\xi_i = \xi_i(x^i)$. 
\end{thm} 

In analogy with (\ref{Poissbrack}), a symmetric contravariant two-tensor $K=(K^{ij})$ is said to be a conformal Killing tensor for the contravariant metric $G=(G^{ij})$ if there exists a vector field $C$ such that
$$
 \{ P_G,P_K\} = 2P_CP_G , 
$$
or equivalently
$$
\nabla^{(i}K^{jk)}=C^{(i}G^{jk)}\,.
$$
Then we have the following intrinsic characterization of the separability of the null HJ and Laplace equations.

\begin{thm}[Kalnins-Miller \cite{KM1983}, Benenti-Chanu-Rastelli \cite{BCR2005}, Chanu-Rastelli \cite{CR2006}]
1) The null HJ equation is separable in orthogonal coordinates if and only if there exists a conformal Killing tensor with simple eigenvalues and normal eigenvectors. \\
2) The null HJ equation is separable in orthogonal coordinates if and only there exist $n$ conformal Killing tensors $K_a, \ a=1,\dots,n$ pointwise independent, with common eigenvectors and in involution. \\ %Moreover, the contravariant metric $G=G^{ij}$ belongs to the algebra generated by the conformal Killing tensors $K_a$ and can be chosen equal to $K_1$. \\ 
3) The Laplace equation is separable in orthogonal coordinates if and only if there exist $n$ conformal Killing tensors $K_a, \ a=1,\dots,n$ pointwise independent, with common eigenvectors, in involution and such that the generalized Robertson condition (\ref{GenRob}) is satisfied. 
\end{thm}

\begin{rem}
  Observe that there does not seem to exist in the literature an intrinsic characterization for the generalized Robertson condition as is the case for St\"ackel metrics in terms of commutation property with the Ricci tensor, see Theorem \ref{Eisenhart} or Theorem \ref{HJH}, 3)). 
\end{rem}

Finally, we can prove that the Laplace-Beltrami operator $\Delta_G$ possesses $n-1$ conformal symmetry operators. For all $a=1,\dots,n$, let us associate to the Killing tensors $K_a$ corresponding to the St\"ackel metric $g = c^{-4} G$, the second-order operators
$$
  H_a : = \Delta_{K_a} - \frac{1}{R} \Delta_{K_a} R,
$$
where the pseudo-Laplacian are defined by
$$
  \Delta_{K_a} := \nabla_i(K_a^{ii} \nabla_i) =  \sum_{i=1}^n K_a^{ii} ( \partial^2_{ii} - \Gamma_i \partial_i). 
$$
Notice that $H_1 = \Delta_g$. Moreover, we say that an operator $H$ is a \emph{conformal symmetry operator} for $\Delta_G$ if $H \psi = 0$ if
$$
  [H,\Delta_G] = L \Delta_G,
$$
for some first-order operator $L$. Then we have

\begin{thm}[Chanu-Rastelli \cite{CR2006}]
  The operators $H_a, \ a=1,\dots,n$ pairwise commute, \textit{i.e.} for all $a,b$
$$
  [H_a, H_b] = 0. 		
$$
Moreover they are conformal symmetry operators for the Laplace-Beltrami operator $\Delta_G$. 
\end{thm}

% OUR 3D MODEL, THE CONSTRUCTION OF THE DN MAP AND ITS STRUCTURE

\subsection{3D conformally St\"ackel manifolds and the structure of the DN map} \label{II2} 

After this quick review on the separability properties of the HJ and Helmholtz equations on conformally St\"ackel manifolds, we would like to specialize the procedure to the three-dimensional case for the Laplace equation that appears in the construction of the DN map and make explicit the \emph{completeness} of the set of separated solutions. \\

% OUR 3D MODEL

\noindent \textit{Three-dimensional conformally St\"ackel cylinders}. Assume that 
$$
  \Omega = [0,A] \times \mathbb{T}^2, 
$$
is a toric cylinder and denote by $x = (x^1,x^2,x^3)$ a global coordinate ssystem on $\Omega$. We consider a Riemannian metric $G$ on $\Omega$ given by
\begin{equation} \label{Metric}
  G = c^4 \, g = \sum_{i=1}^{3} H_{i}^{2}(dx^{i})^{2}\,,
\end{equation}
with $g$ a St\"ackel metric
$$
  g = \sum_{i=1}^{3}{h}_{i}^{2}(dx^{i})^{2}\,, \quad h_{i}^{2}=\frac{\det S}{s^{i1}}\,,
$$
with $S = (s_{ij}(x^i))$ being a non-singular St\"ackel matrix. In order to insure $R$-separability of the Laplace equation, we assume that the generalized Robertson condition (\ref{GenRob}) holds. Writing this equation in coordinates and remarking that $\Gamma_i = \gamma_i - 2\partial_i \ln c$, the generalized Robertson condition is seen to be equivalent to the PDE (\ref{PDEc}) on the conformal factor $c$, \textit{i.e.}
$$
  -\Delta_{g}c - \sum_{i=1}^3 h_{i}^{2}\big(\phi_{i}+\frac{1}{4}\gamma_{i}^{2}-\frac{1}{2}\partial_{i}\gamma_{i}\big)c = 0\,.
$$

Let us calculate now the Laplace equation in this coordinate system and see how variables separation naturally appears in the calculations. We start with
$$
  -\Delta_G \, \psi = 0,
$$ 
and we look for a solution $\psi$ under the form $\psi = R u$. Then $u$ must satisfy
$$
  -\Delta_G \, u - \frac{2}{R} G^{-1}(dR,du) - \frac{\Delta_G \, R}{R} u = 0,
$$
or equivalently in the conformally separable St\"ackel coordinates $x=(x^i)$
\begin{equation} \label{a1}
  \sum_{i=1}^3 \,\left[ H_i^{-2} (-\partial_{ii}^2 u + \Gamma_i \partial_i u) - 2H_i^{-2} \partial_i \ln R \,\partial_i u \right] - \frac{\Delta_G \, R}{R} u = 0. 
\end{equation}
Choose the $R$ factor so as to satisfy (\ref{RDef}), \textit{i.e.} $2\partial_i \ln R = \Gamma_i$. (Since the functions $\xi_i$ appearing in (\ref{RDef}) are arbitrary in (\ref{RDef}), we choose them to be zero for convenience. )

\begin{rem}
Recall that
\begin{eqnarray*}
  \Gamma_i & := & -\partial_i \ln \frac{H_1 H_2 H_3}{H_i^2} = -\partial_i \ln \frac{c^2 h_1 h_2 h_3}{h_i^2} = -\partial_i \ln \frac{c^2 \sqrt{\det S} s^{i1}}{\sqrt{ s^{11} s^{12} s^{13}}}, \\
	         & = & -\partial_i \ln \frac{c^2 \sqrt{\det S}}{\sqrt{ s^{11} s^{12} s^{13}}}, \quad \textrm{since} \ \ \partial_i \, s^{i1} = 0. 
\end{eqnarray*}
Comparing with (\ref{RDef}), we see that $R$ may be written as 
\begin{equation} \label{R}
  R = \left( \frac{s^{11} s^{21} s^{31}}{c^4 \, \det S} \right)^{\frac{1}{4}}.
\end{equation}
\end{rem}

\noindent Under the assumption (\ref{RDef}) or equivalenty (\ref{R}), we calculate
\begin{equation} \label{a2}
  \frac{\Delta_G \, R}{R} = \frac{1}{4} H_i^{-2} \left( 2 \partial_i \Gamma_i - \Gamma_i^2 \right). 
\end{equation}
Putting together (\ref{RDef}), (\ref{a1}) and (\ref{a2}), we get the following expression for the Laplace equation
$$
  \sum_{i=1}^3 \, H_i^{-2} \left[ -\partial_{ii}^2 u - \frac{1}{4} H_i^{-2} \left( 2 \partial_i \Gamma_i - \Gamma_i^2 \right) \right] u = 0. 
$$
Finally, using (\ref{GenRob}), we obtain 
\begin{equation} \label{a3}
  \sum_{i=1}^3 \, H_i^{-2} \left[ -\partial_{ii}^2 u - \phi_i(x^i) \right] u = 0.
\end{equation}

Let us introduce the ordinary differential operators $A_i = -\partial_{ii}^2 - \phi_i(x^i)$ from (\ref{A}). Then we continue the separation of variables procedure as follows
\begin{eqnarray}
  -\Delta_G \, \psi & \Longleftrightarrow & \sum_{i=1}^3 \, H_i^{-2} A_i u = 0, \nonumber \\  
	                  & \Longleftrightarrow & A_1 u + \frac{H_1^2}{H_2^2} A_2 u + \frac{H_1^2}{H_3^2} A_3 u = 0, \nonumber \\
										& \Longleftrightarrow & A_1 u + s_{12}(x^1) H u + s_{13}(x^1) L u = 0,  \label{a4}
\end{eqnarray}
where the operators $(H,L)$ are given by (\ref{HL}), \textit{i.e.}
$$
  \left( \begin{array}{c} H \\ L \end{array} \right) = \frac{1}{s^{11}} \left( \begin{array}{cc} -s_{33} & s_{23} \\ s_{32} & -s_{22} \end{array} \right) \left( \begin{array}{c} A_2 \\ A_3 \end{array} \right),
$$
and are computed thanks to the St\"ackel structure of $g$. We recall the following Lemma from Gobin \cite{Go2018} (Lemma 2.5, Remarks 2.6 and 2.7).

\begin{lemma}
The operators $H$ and $L$ are elliptic selfadjoint operators on $L^2(\mathbb{T}^2; \, s^{11} dx^1 dx^2)$ that commute, \textit{i.e.} $[H,L] = 0$. The basis of common eigenfunctions $(Y_m)_{m \geq 1}$, with joint spectrum denoted by $(\mu_m^2, \nu_m^2)$, \textit{i.e.}
$$
  H Y_m = \mu_m^2 Y_m, \quad L Y_m = \nu_m^{2}Y_{m}, 
$$
can be written as $Y_m = v_m(x^2) w_m(x^3)$ and satisfy
$$
	L^2(\mathbb{T}^2; \, s^{11}\, dx^2 dx^3) = \bigoplus_m \langle Y_m \rangle.   
$$	
\end{lemma}

\begin{rem}
Note that since $s^{11} \ne 0$, we have  
\begin{equation} \label{a5}
  \left( \begin{array}{c} A_2 \\ A_3 \end{array} \right) = - \left( \begin{array}{cc} s_{22} & s_{23} \\ s_{32} & s_{33} \end{array} \right) \left( \begin{array}{c} H \\ L \end{array} \right). 
\end{equation}
\end{rem}

\noindent We now finish the separation of variables procedure by looking for the solutions $\psi$ of $-\Delta_G\,\psi = 0$ under the form 
\begin{equation} \label{a6}
  \psi = R \sum_{m=1}^\infty \, u_m(x^1) Y_m, \quad Y_m = v_m(x^2) w_m(x^3). 
\end{equation}
Putting (\ref{a6}) into (\ref{a4}) and (\ref{a5}), we deduce that $u_m, \, v_m, \, w_m$ satisfy the three separated ODEs (\ref{Separateda}) - (\ref{Separatedc}), \textit{i.e.} 
\begin{eqnarray*}
	-u_m'' + [\,\mu_m^2 s_{12}(x^1) + \nu_m^2 s_{13}(x^1) - \phi_1(x^1)] \, u_m & = & 0, \\ 
	-v_m'' + [\,\mu_m^2 s_{22}(x^2) + \nu_m^2 s_{23}(x^2) - \phi_2(x^2)] \, v_m & = & 0, \\ 
	-w_m'' + [\,\mu_m^2 s_{32}(x^3) + \nu_m^2 s_{33}(x^3) - \phi_3(x^3)] \, w_m & = & 0. 
\end{eqnarray*}
This finishes the procedure of variables separation for the Laplace equation on a conformally St\"ackel manifold. \\

\noindent \textit{Some hidden invariances}. When solving the inverse Calder\'on problem on conformally St\"ackel manifolds in Section \ref{Inverse}, we will need to understand some underlying invariances in the definition of our metrics and in the procedure of variables separation. The first and main invariance comes from the fact that a S\"tackel metric $g$ as in (\ref{MetricStackel}) is \emph{not} determined by a unique St\"ackel matrix $S$. Precisely, we quote the following Proposition from Gobin \cite{Go2018}

\begin{prop} \label{InvariancesStackel}
Let $S$ be a St\"ackel matrix and $g_S$ the corresponding St\"ackel metric. \\
1. Let $G \in GL_2(\R)$ a constant matrix and define the new St\"ackel matrix
$$ 
  \hat{S} =  \begin{pmatrix} s_{11}(x^1) & \hat{s}_{12}(x^1) & \hat{s}_{13}(x^1) \\ s_{21}(x^2) & \hat{s}_{22}(x^2) & \hat{s}_{23}(x^2)  \\ s_{31}(x^3) & \hat{s}_{32}(x^3) & \hat{s}_{33}(x^3) \end{pmatrix},
$$	
satisfying
$$
  \begin{pmatrix} s_{i2} & s_{i3} \end{pmatrix} = \begin{pmatrix} \hat{s}_{i2} & \hat{s}_{i3} \end{pmatrix} \,G, \quad \forall i \in \{1,2,3\}.
$$	
Then $g_{\hat{S}} = g_S$. \\
2) Define the St\"ackel matrix
$$
  \hat{S} =  \begin{pmatrix} \hat{s}_{11}(x^1) & s_{12}(x^1) & s_{13}(x^1) \\ \hat{s}_{21}(x^2) & s_{22}(x^2) & s_{23}(x^2)  \\ \hat{s}_{31}(x^3) & s_{32}(x^3) & s_{33}(x^3) \end{pmatrix},
$$	
where,
$$
 \begin{cases}
 \hat{s}_{11}(x^1) = s_{11}(x^1) + C_1 s_{12}(x^1) + C_2 s_{13}(x^1) \\
 \hat{s}_{21}(x^2) = s_{21}(x^2) + C_1 s_{22}(x^2) + C_2 s_{23}(x^2) \\
 \hat{s}_{31}(x^3) = s_{31}(x^3) + C_1 s_{32}(x^3) + C_2 s_{33}(x^3)
\end{cases},
$$
where $C_1$ and $C_2$ are real constants. Then $g_{\hat{S}} = g_S$.
\end{prop}

These invariances are important and will naturally appear at several stages in our proof of uniqueness in the inverse problem. Note moreover that these invariances allow us to assume from the very beginning and without loss of generality certain properties for the St\"ackel matrix $S$ used to represent a given St\"ackel metric $g_S$. Precisely, we have

\begin{prop}[Gobin \cite{Go2018}, Prop. 1.17 and Remark 1.18] \label{StackelForm}
Using the above invariances, we can always choose a St\"ackel matrix $S$ associated to a Riemannian St\"ackel metric $g_S$ such that
$$
 \begin{cases}
 \hat{s}_{12}(x^1) > 0 \quad \textrm{and} \quad \hat{s}_{13}(x^1) > 0, \quad \forall x^1\\
 \hat{s}_{22}(x^2) < 0 \quad \textrm{and} \quad \hat{s}_{23}(x^2) > 0, \quad \forall x^2\\
 \hat{s}_{32}(x^3) > 0 \quad \textrm{and} \quad \hat{s}_{33}(x^3) < 0, \quad \forall x^3\\
 %\lim\limits_{x^1 \to 0} s_{12}(x^1) = \lim\limits_{x^1 \to 0} s_{13}(x^1) = 1
\end{cases}.
$$	
As a consequence, we can always assume from the very beginning
$$
  s^{11}, \ s^{21}, \ s^{31}, \ \det S > 0.
$$
Finally, note that
$$
  s^{11}> 0 \quad \Longleftrightarrow \quad \frac{s_{22}}{s_{23}} < \frac{s_{32}}{s_{33}}, \quad \forall x^2, x^3. 
$$
\end{prop}

There is a second and last invariance that we need to understand before attacking the inverse problem. This invariance appears in the procedure of variables separation when we look for solutions of the Laplace equation decomposed onto the Hilbert basis of angular harmonics $Y_m$ which are common eigenfunctions of the operators $(H,L)$. This decomposition is not unique since we could for example have decomposed the solutions onto the Hilbert basis of common eigenfunctions of the operators 
$$
  \hat{H} = H + B_1, \quad \hat{L} = L + B_2,
$$
where $B_1, B_2$ are two constants. The common eigenfunctions, which are still the $Y_m$, would be then associated to the joint spectrum
\begin{equation} \label{InvJointSpectrum}
  \hat{\mu}_m^2 = \mu_m^2 + B_1, \quad \hat{\nu}_m^2 = \nu_m^2 + B_2. 
\end{equation}
Then the separability procedure would remain unaffected and would still lead to the separated ODEs (\ref{Separateda}) - (\ref{Separatedc}) with the only modification (\ref{InvJointSpectrum}). We thus have the freedom to choose the constants $B_1, B_2$ as we wish in the separated equations. This invariance will be important at one point in Section \ref{III2}. \\

\noindent \textit{The construction of the DN map and its structure}. We construct here the DN map associated to a conformally St\"ackel manifold $(M,G)$. Consider the Dirichlet problem
\begin{equation} \label{DP}
  \left\{ \begin{array}{rcl} -\Delta_G \, \psi & = & 0, \ \textrm{on} \ \Omega, \\
	                                       \psi & = & f, \ \textrm{on} \ \partial \Omega.
	\end{array} \right.																			
\end{equation}
Recall that $\partial \Omega = \Omega_0 \cup \Omega_1$ where $\Omega_j \simeq \mathbb{T}^2$. Hence we can identify the Dirichlet data $f \in H^{\frac{1}{2}}(\partial \Omega)$ with the two-component vector
$$
  f = \left( \begin{array}{c} f_0 \\ f_1 \end{array} \right) \in H^{\frac{1}{2}}(\Omega_0) \oplus H^{\frac{1}{2}}(\Omega_1).
$$
By definition, the DN map is given by
$$
  \Lambda_G \, f := (\partial_\nu \psi )_{|\partial \Omega} = \left( \begin{array}{c} \left( \partial_{\nu_0} \psi \right)_{| \Omega_0} \\ \left( \partial_{\nu_1} \psi \right)_{| \Omega_1} \end{array} \right)
$$
where $\psi$ is the unique solution of (\ref{DP}) and $\nu_0$ and $\nu_1$ are the outgoing unit normal vectors on $\Omega_0$ and $\Omega_1$ respectively. A short calculation using the form (\ref{Metric}) of the metric $G$ leads to
$$
  \Lambda_G \, f = \left( \begin{array}{c} \left( -\frac{1}{H_1} \partial_{x^1} \psi \right)_{|\,x^1 = 0} \\ \left( \frac{1}{H_1} \partial_{x^1} \psi \right)_{|\,x^1 = A} \end{array} \right).
$$
Recalling that $\psi = R \, u$ with the $R$-factor given by (\ref{RDef}) or explicitly by (\ref{R}), we get
$$
  \Lambda_G \, f = \left( \begin{array}{c} -\frac{R}{H_1} \left[ (\partial_{1} \ln R) u + \partial_1 u \right]_{|\,x^1 = 0} \\ \frac{R}{H_1} \left[ (\partial_{1} \ln R) u + \partial_1 u \right]_{|\,x^1 = A} \end{array} \right).
$$
But we know from (\ref{RDef}) that $\partial_1 \ln R = \frac{1}{2} \Gamma_1$. Hence 
$$
  \Lambda_G \, f = \left( \begin{array}{c} -\frac{R}{H_1} \left[ \frac{1}{2} \Gamma_1 u + \partial_1 u \right]_{|\,x^1 = 0} \\ \frac{R}{H_1} \left[ \frac{1}{2} \Gamma_1 u + \partial_1 u \right]_{|\,x^1 = A} \end{array} \right).
$$
Let us use at this point the separated form (\ref{a6}) of the solution $\psi$, \textit{i.e.}
$$
  \psi = R\, u, \quad u = \sum_{m \geq 1}^\infty u_m(x^1) \, Y_m, 
$$
and the corresponding Fourier decomposition of the Dirichlet data $f$ on $\partial \Omega$ :
$$
  f = R \, \varphi, \quad \varphi = \left( \begin{array}{c} \varphi^0 \\ \varphi^1 \end{array} \right), \quad \varphi^j = \sum_{m \geq 1}^\infty \varphi_m^j \, Y_m, \quad j=0,1. 
$$ 
We observe that the functions $u_m(x^1)$ satisfy the one-dimensional Dirichlet problem on $[0,A]$ : 
$$
  \left\{ \begin{array}{rcl} 
	-u_m'' + [\mu_m^2 s_{12}(x^1) + \nu_m^2 s_{13}(x^1) - \phi_1(x^1)] u_m & = & 0, \\
	u_m(0) = \varphi_m^0, \ u_m(A) = \varphi_m^1. & &
	\end{array} \right.
$$
We thus obtain the following decomposition for the DN map $\Lambda_G$ 
\begin{equation} \label{b0}
  \Lambda_G \, f = \sum_{m \geq 1}^\infty \left( \begin{array}{c} -\frac{R(0,x^2,x^3)}{H_1(0,x^2,x^3)} \left[ \frac{1}{2} \Gamma_1(0,x^2,x^3) \varphi_m^0 + u_m'(0) \right] \\ \frac{R(A,x^2,x^3)}{H_1(A,x^2,x^3)} \left[ \frac{1}{2} \Gamma_1(A,x^2,x^3) \varphi_m^1 + u_m'(A) \right] \end{array} \right) \, Y_m.
\end{equation}
It remains essentially to express the derivatives $u_m'(0)$ and $u_m'(A)$ in terms of the Dirichlet data $\varphi_m^0$ and $\varphi_m^1$. This can be done as follows. \\

Denote by $\{c_0,s_0\}$ and $\{c_1,s_1\}$ the fundamental systems of solutions (FSS) of the separated ODE
\begin{equation} \label{RadialODE}
  -u'' + [\mu^2 s_{12}(x^1) + \nu^2 s_{13}(x^1) - \phi_1(x^1)] u = 0,
\end{equation}
(where $\mu^2$ and $\nu^2$ are here any constants) which satisfy the Cauchy conditions of sine and cosine type at $x^1 = 0$ and $x^1 = A$, \textit{i.e.}
\begin{eqnarray*}
  c_0(0) = 1, \ c_0'(0) = 0, \ s_0(0) = 0, \ s_0'(0) = 1, \\
	c_1(A) = 1, \ c_1'(A) = 0, \ s_1(A) = 0, \ s_1'(A) = 1. 
\end{eqnarray*}
Clearly, the functions $c_j, \, s_j, \ j=0,1$ are analytic separately in the parameters $\mu, \nu \in \C$ and their Wronskians satisfy
$$
  W(c_j, s_j) = 1, \ j=0,1, 
$$
where $W(f,g) := fg' - f'g$. 
 
Associated to the ODE (\ref{RadialODE}) with Dirichlet boundary conditions, we introduce first the characteristic function
\begin{equation} \label{Char}
  \Delta(\mu^2,\nu^2) = W(s_0,s_1).
\end{equation}
Second, introduce the Weyl solutions of (\ref{RadialODE}) given by the particular linear combinations 
$$
  \Psi = c_0 + M(\mu^2,\nu^2) s_0, \quad \Phi = c_1 - N(\mu^2,\nu^2) s_1,
$$
by demanding that they satisfy the Dirichlet boundary condition at $x=A$ and $x=0$ respectively. The coefficients $M,N$ are the Weyl-Titchmarsh (WT) functions and one easily verifies that they can be expressed as follows in terms of the Wronskians and characteristic functions, 
\begin{equation} \label{WT}
  M(\mu^2,\nu^2) = -\frac{W(c_0,s_1)}{\Delta(\mu^2,\nu^2)} = -\frac{D(\mu^2,\nu^2)}{\Delta(\mu^2,\nu^2)}, \quad \quad N(\mu^2,\nu^2) = \frac{W(s_0,c_1)}{\Delta(\mu^2,\nu^2)} = \frac{E(\mu^2,\nu^2)}{\Delta(\mu^2,\nu^2)}.
\end{equation}
Finally it is an easy calculation to show that
\begin{equation} \label{b1}
  \begin{array}{c} u_m'(0) = M(\mu_m^2,\nu_m^2) \, \varphi_m^0 + \frac{1}{\Delta(\mu_m^2,\nu_m^2)} \, \varphi_m^1, \\
	u_m'(A) = \frac{1}{\Delta(\mu_m^2,\nu_m^2)} \, \varphi_m^0 + N(\mu_m^2,\nu_m^2) \, \varphi_m^1.
	\end{array} 
\end{equation}

Coming back to the expression of the DN map, we obtain from (\ref{b0}) and (\ref{b1}) the following expression 
$$
  \Lambda_G \, f = \sum_{m \geq 1}^\infty \left( \begin{array}{cc} -\frac{R(0)}{H_1(0)} \left[ \frac{1}{2} \Gamma_1(0) + M(\mu_m^2,\nu_m^2) \right] &  -\frac{R(0)}{H_1(0)} \frac{1}{\Delta(\mu_m^2,\nu_m^2)} \\ \frac{R(A)}{H_1(A)} \frac{1}{\Delta(\mu_m^2,\nu_m^2)} & \frac{R(A)}{H_1(A)} \left[ \frac{1}{2} \Gamma_1(A) + N(\mu_m^2,\nu_m^2) \right] \end{array} \right) \,\left( \begin{array}{c} \varphi_m^0 \\ \varphi_m^1 \end{array} \right) \, Y_m,
$$
where we used the notations
\begin{eqnarray*}
  R(0) = R(0,x^2,x^3), \ H_1(0) = H_1(0,x^2,x^3), \ \Gamma_1(0) = \Gamma_1(0,x^2,x^3), \\
	R(A) = R(A,x^2,x^3), \ H_1(A) = H_1(A,x^2,x^3), \ \Gamma_1(A) = \Gamma_1(A,x^2,x^3),
\end{eqnarray*}
as well as a $2 \times 2$-matrix valued notation for the DN map on each harmonic $Y_m$. A last manipulation of the above expression of the DN map leads to the form (\ref{DNStructure}) - (\ref{AG}) announced in the Introduction, \textit{i.e.}
$$
  \Lambda_G = \left( \begin{array}{cc} \frac{-1}{H_1(0)} & 0 \\ 0 & \frac{1}{H_1(A)} \end{array} \right) \Bigg[ 
	            \left( \begin{array}{cc} \frac{\Gamma_1(0)}{2} & 0 \\ 0 & \frac{\Gamma_1(A)}{2} \end{array} \right) + \ \left( \begin{array}{cc} R(0) & 0 \\ 0 & R(A) \end{array} \right) A_G \left( \begin{array}{cc} \frac{1}{R(0)} & 0 \\ 0 & \frac{1}{R(A)} \end{array} \right) \Bigg]
$$
where 
$$
  A_G = \bigoplus_{m \geq 1} A_G^m, \quad \quad A_G^m:= (A_G)_{|\langle Y_m \rangle} := \left( \begin{array}{cc} M(\mu_m^2,\nu_m^2)  & \frac{1}{\Delta(\mu_m^2,\nu_m^2)} \\ \frac{1}{\Delta(\mu_m^2,\nu_m^2)} & N(\mu_m^2,\nu_m^2) \end{array} \right).
$$

\vspace{0.5cm}
\noindent From the above structure of the DN map, we make two comments : \\

1) The full DN map $\Lambda_G$ is essentially encoded in the operator $A_G$ which is diagonalizable on the Hilbert basis $(Y_m)_{m \geq 1}$. The radial part of the metric (\textit{i.e.} the functions depending on $x^1$) appears there in the definition of the characteristic and Weyl-Titchmarsh functions (\ref{Char}) and (\ref{WT}). On the contrary, the angular part of the metric (\textit{i.e.} the functions depending on $(x^2, \, x^3)$) appears in the joint spectrum $J = \{(\mu_m^2, \nu_m^2), \ m \geq 1\}$ at which the characteristic and WT functions are evaluated. The inverse problem for the operator $A_G$ will be studied thanks to the multi-parameter CAM method in Section \ref{III3}. \\

2) The full DN map differs from the operator $A_G$ by explicit boundary values of the essential functions $H_1$, $R$ and $\Gamma_1$ which are a priori unknown in the inverse problem. These boundary values will be uniquely determined however thanks to usual boundary determination results in Section \ref{III2}.

%%%%%%%%%%%%%%%%%%%%%%%%%%%%%%%%%%%%%%%%% THE INVERSE PROBLEM %%%%%%%%%%%%%%%%%%%%%%%%%%%%%%%%%%%%%%%%%%%%%%%%%%%%

\Section{The Calder\'on inverse problem} \label{Inverse}

% REDUCTION TO THE WHOLE CYLINDER

\subsection{Reduction to an inverse problem on the whole cylinder $\Omega$} \label{III1}

Recall that we consider smooth compact connected Riemannian manifolds $(M,G)$ and $(M,\tilde{G})$ such that : \\

1) $ M \subset \subset \Omega = [0,A] \times \mathbb{T}^2$. \\

2) $G$ and $\tilde{G}$ are Riemannian metrics on $M$ having the form (\ref{MetricConf}) - (\ref{StackelMatrix}). \\

3) $\Lambda_G = \Lambda_{\tilde{G}}$. \\ 

\noindent We aim to show in this section that it is enough to prove uniqueness in the inverse Calder\'on problem for conformally St\"ackel metrics $G$ and $\tilde{G}$ \emph{on the whole cylinder} $\Omega$ on which we will be able to use separation of variables. This will be done using the extension procedure explained in the Introduction, that is by proving Theorem \ref{ExtensionThm} and its consequence Proposition \ref{ExtensionCalderon}.

\begin{proof}[Proof of Theorem \ref{ExtensionThm}] 
We assume that 
$$
  \Lambda_{G,1} = \Lambda_{\tilde{G}, 1}.
$$
From the boundary determination results in \cite{LeU1989, KY2002}, recall that the metrics $G$ and $\tilde{G}$ coincide on $\partial M_1$ as well as all their normal derivatives. Hence we can identify the normal derivative operators $\partial_\nu$ and $\partial_{\tilde{\nu}}$ on $\partial M_1$. 

Let $f \in H^{\frac{1}{2}}(M_2)$ and consider  the unique solutions $u,\,\tilde{u}$ of the Dirichlet problems
$$
  \left\{ \begin{array}{rcl} 
	-\Delta_G\, u & = & 0, \quad \textrm{on} \ M_2, \\
	u & = & f, \quad \textrm{on} \ \partial M_2,
	\end{array} \right. 
	\quad \quad 
	\left\{ \begin{array}{rcl} 
	-\Delta_{\tilde{G}}\, \tilde{u} & = & 0, \quad \textrm{on} \ M_2, \\
	\tilde{u} & = & f, \quad \textrm{on} \ \partial M_2.
	\end{array} \right.	
$$
Since $G = \tilde{G}$ on $M_2 \setminus M_1$, we aim to show that $\partial_\nu u = \partial_{\nu} \tilde{u}$ on $\partial M_2$. 

For this, introduce the solution $u_{in}$ of the Dirichlet problem
$$
  \left\{ \begin{array}{rcl} 
	-\Delta_{\tilde{G}}\, u_{in} & = & 0, \quad \textrm{on} \ M_1, \\
	u_{in} & = & \psi, \quad \textrm{on} \ \partial M_1,
	\end{array} \right. 
$$	
where $\psi = u_{| \partial M_1}$ and define the function
$$
  v := \left\{ \begin{array}{l} 
	u_{in} \ \textrm{on} \ M_1, \\
	u \ \textrm{on} \ M_2 \setminus M_1. 
	\end{array} \right. 	
$$	
Since $G = \tilde{G}$ on $M_2 \setminus M_1$, we clearly have
$$
  \Delta_{\tilde{G}} \,v = 0, \quad \textrm{on} \ M_1 \cup (M_2 \setminus M_1),
$$
and $v = u$ on $\partial M_1$ by definition of $u_{in}$. Let us study the traces of the normal derivatives of $v$ at the interface $\partial M_1$ when the normal derivatives are taken from the exterior and the interior. We have for the former
$$
  \partial_\nu v_{| \partial M_1^+} = \partial_\nu u_{| \partial M_1^+} = \Lambda_{G,1} \psi,  
$$ 
and for the latter
$$
  \partial_\nu v_{| \partial M_1^-} = \partial_\nu (u_{in})_{| \partial M_1^-} = \Lambda_{\tilde{G},1} \psi = \Lambda_{G,1} \psi,
$$
thanks to our main hypothesis. We deduce that $v$ and $\partial_\nu v$ are continuous on $\partial M_1$ and thus
$$
  \left\{ \begin{array}{rcl} 
	-\Delta_{\tilde{G}}\, v & = & 0, \quad \textrm{on} \ M_2, \\
	v & = & f, \quad \textrm{on} \ \partial M_2.
	\end{array} \right.
$$
By uniqueness of the Dirichlet problem, we infer that $v = \tilde{u}$ on $M_2$. This implies that $u = \tilde{u}$ on $M_2 \setminus M_1$ and therefore $\partial_\nu u = \partial_\nu \tilde{u}$ on $\partial M_2$, whence
$$
  \Lambda_{G,2} = \Lambda_{\tilde{G},2}. 
$$	
\end{proof}

Let us apply now this extension result to the inverse Calder\'on problem on conformally St\"ackel manifolds. From the boundary determination results of \cite{LeU1989, KY2002}, we know that there exists a neighbourhood $U$ of $\partial M$ and a diffeomorphism $\phi \in $Diff$(U)$ such that 
$$
  \sum_{|\alpha|=0}^N  \sup_{x \in \partial M} | \partial^\alpha (\tilde{G}(x) - \phi^* G(x)) | = 0, \quad \forall N \geq 0. 
$$
In particular, the metrics $G$ and $\tilde{G}$ coincide on $\partial M$ as well as all their tangential and normal derivatives at any order $N$. We finish the extension procedure in the the following way:

\begin{proof}[Proof of Proposition \ref{ExtensionCalderon}] 
We first extend the metric $G$ on $M$ to a (still) conformally St\"ackel metric to the whole cylinder $\Omega$ and we demand that $G$ satisfies the generic assumption (\ref{Generic}). Then we extend the metric $\tilde{G}$ on $M$ to a conformally St\"ackel metric $\hat{G}$ on $\Omega$ by defining
$$
  \hat{G} = \left\{ \begin{array}{l} \tilde{G} = \tilde{G}, \quad \textrm{on} \ M, \\
	                                     \tilde{G} = G, \quad \textrm{on} \ \Omega \setminus M.
							\end{array} \right. 												 
$$ 
The new metric $\hat{G}$ on $\Omega$ is smooth thanks to the above boundary determination results, is clearly conformally St\"ackel since $G$ and $\tilde{G}$ are, and satisfies
$$
  \hat{G} = G, \quad \textrm{on} \ \Omega \setminus M.
$$
Hence Theorem \ref{ExtensionThm} implies that
$$
  \Lambda_{G,\Omega} = \Lambda_{\hat{G},\Omega}.  
$$
\end{proof}

% BOUNDARY DETERMINATION

\subsection{Boundary determination results} \label{III2}

Using the results of Section \ref{III1}, we are led to study the Calder\'on problem on conformally St\"ackel \emph{cylinders}. Precisely, we consider two smooth compact connected Riemannian manifolds $(\Omega,G)$ and $(\tilde{\Omega},\tilde{G})$ such that : \\

1) $ \Omega = [0,A] \times \mathbb{T}^2$ and $ \tilde{\Omega} = [0, \tilde{A}] \times \mathbb{T}^2$ are toric cylinders. \\

2) The Riemannian metrics $G$ and $\tilde{G}$ on $\Omega$ have the conformally St\"ackel form (\ref{MetricConf}) - (\ref{StackelMatrix}) and satisfy the generic assumption (\ref{Generic}). \\

3) Their DN maps coincide, that is $\Lambda_G = \Lambda_{\tilde{G}}$ \footnote{Note that we identify the boundaries $\partial \Omega = \partial \tilde{\Omega}$ when stating this equality.}. \\

\noindent In this Section, we use the boundary determination results of \cite{LeU1989, KY2002} to obtain the maximum of informations on the metrics $G$ and $\tilde{G}$. Precisely we will use the facts that the metrics $G$ and $\tilde{G}$ and their normal derivatives coincide on $\partial \Omega = \partial \tilde{\Omega}$. We divide our boundary determination results into three steps. \\

\noindent \textit{Step 1}. Assume first that 
\begin{equation} \label{c0}
  G = \tilde{G}, \quad \textrm{on} \ \partial \Omega.
\end{equation}
Recall that $\Omega = \Omega_0 \cup \Omega_1$ and observe that
\begin{equation} \label{c0bis}
  \left\{ \begin{array}{c} 
	G_{|\Omega_0} = H_2^2(0,x^2,x^3) (dx^2)^2 + H_3^2(0,x^2,x^3) (dx^3)^2, \\
	G_{|\Omega_1} = H_2^2(A,x^2,x^3) (dx^2)^2 + H_3^2(A,x^2,x^3) (dx^3)^2.
	\end{array} \right. 
\end{equation}
Hence we get from (\ref{c0}) and (\ref{c0bis}) that
\begin{equation} \label{c1}
  H_j(0) = \tilde{H}_j(0), \quad H_j(A) = \tilde{H}_j(\tilde{A}), \quad j=2,3,
\end{equation}
where we used the shorthand notations \footnote{More generally, in this Section, given a function $f = f(x^1,x^2,x^3)$, we use the notations $f(0)$ and $f(A)$ for $f(0,x^2,x^3)$ and $f(A,x^2,x^3)$ respectively.}
$$
  H_j(0) = H_j(0,x^2,x^3), \quad H_j(A) = H_j(A,x^2,x^3), \quad j=2,3. 
$$
In particular, using the the definition of the diagonal coefficients $H_j$ given by (\ref{MetricConf}) and (\ref{MetricStackel}), we have at $x^1 = 0$ 
\begin{eqnarray}
  \frac{H_2^2(0)}{H_3^2(0)} = \frac{\tilde{H}_2^2(0)}{\tilde{H}_3^2(0)} & \Longleftrightarrow & \frac{h_2^2(0)}{h_3^2(0)} = \frac{\tilde{h}_2^2(0)}{\tilde{h}_3^2(0)}, \nonumber \\
	& \Longleftrightarrow & \frac{s^{31}(0)}{s^{21}(0)} = \frac{\tilde{s}^{31}(0)}{\tilde{s}^{21}(0)}, \nonumber \\
	& \Longleftrightarrow & \frac{s^{31}(0)}{\tilde{s}^{31}(0)} = \frac{s^{21}(0)}{\tilde{s}^{21}(0)}. \label{c2}
\end{eqnarray}
The remarkable properties of St\"ackel metrics manifest themselves here since the functions $\frac{s^{31}(0)}{\tilde{s}^{31}(0)}$ and $\frac{s^{21}(0)}{\tilde{s}^{21}(0)}$ appearing in (\ref{c2}) only depend on $x^2$ and $x^3$ respectively. We thus infer from (\ref{c2}) that there exists a constant $C_0$ such that 
\begin{equation} \label{c3}
  s^{31}(0) = C_0 \tilde{s}^{31}(0), \quad s^{21}(0) = C_0 \tilde{s}^{21}(0). 
\end{equation}
Similarly, working at $x^1 = A$, we see that there exists a constant $C_1$ such that 
\begin{equation} \label{c3A}
  s^{31}(A) = C_1 \tilde{s}^{31}(\tilde{A}), \quad s^{21}(A) = C_1 \tilde{s}^{21}(\tilde{A}). 
\end{equation}

Recalling that
$$
  s^{21} = s_{13} s_{32} - s_{12}s_{33}, \quad s^{31} = s_{12} s_{23} - s_{13}s_{22},
$$ 
we get from (\ref{c3})
\begin{equation} \label{c4}
  \left( \begin{array}{cc} s_{22} & s_{23} \\ s_{32} & s_{33} \end{array} \right) \left( \begin{array}{c} -s_{13}(0) \\ s_{12}(0) \end{array} \right) = C_0 \left( \begin{array}{cc} \tilde{s}_{22} & \tilde{s}_{23} \\ \tilde{s}_{32} & \tilde{s}_{33} \end{array} \right) \left( \begin{array}{c} -\tilde{s}_{13}(0) \\ \tilde{s}_{12}(0) \end{array} \right).
\end{equation}
Let us introduce the notation
\begin{equation} \label{T}
  T = \left( \begin{array}{cc} s_{22} & s_{23} \\ s_{32} & s_{33} \end{array} \right),
\end{equation}
and observe that $\det T = s^{11} \ne 0$ since $G$ is a Riemannian metric. Hence (\ref{c4}) can be rewritten as
\begin{equation} \label{c5}
  \frac{1}{C_0} \tilde{T}^{-1} T \left( \begin{array}{c} -s_{13}(0) \\ s_{12}(0) \end{array} \right) = \left( \begin{array}{c} -\tilde{s}_{13}(0) \\ \tilde{s}_{12}(0) \end{array} \right).
\end{equation}
In this equality, only the $2\times 2$-matrix $\tilde{T}^{-1} T$ depends on the variables $x^2,x^3$. Hence differentiating (\ref{c5}) with respect to $x^2$ and $x^3$, we obtain
$$
  \left( \partial_j \, \tilde{T}^{-1} T \right) \left( \begin{array}{c} -s_{13}(0) \\ s_{12}(0) \end{array} \right) = \left( \begin{array}{c} 0 \\ 0 \end{array} \right), \quad j=2,3. 
$$
We deduce from this that 
\begin{equation} \label{c6}
  \left( \begin{array}{c} -s_{13}(0) \\ s_{12}(0) \end{array} \right) \in \ker \left( \partial_j \, \tilde{T}^{-1} T \right), \quad j=2,3. 
\end{equation}
Since a similar analysis can be done at $x^1 = A$, we also have
\begin{equation} \label{c6A}
  \left( \begin{array}{c} -s_{13}(A) \\ s_{12}(A) \end{array} \right) \in \ker \left( \partial_j \, \tilde{T}^{-1} T \right), \quad j=2,3. 
\end{equation}
Thanks to our generic hypothesis (\ref{Generic}), we infer from (\ref{c6}) and (\ref{c6A}) that $\dim \ker \left( \partial_j \, \tilde{T}^{-1} T \right) = 2$ and thus
$$
    \partial_j \, \tilde{T}^{-1} T  = 0, \quad j=2,3. 
$$
We deduce from this that there exists a constant invertible matrix $G \in GL_2(\R)$ such that
\begin{equation} \label{c7}
  T = \tilde{T} G.
\end{equation}
Finally, we recall from the hidden invariances stated in Proposition \ref{InvariancesStackel} that we do not alter the St\"ackel metrics $g$ and $\tilde{g}$ by multiplying from the right the last two columns of their St\"ackel matrices by an invertible matrix $G$. Hence we can use this invariance and (\ref{c7}) to assume from now on that
\begin{equation} \label{c8}
  T = \tilde{T}, \quad \textrm{and thus} \quad s^{11} = \tilde{s}^{11}. 
\end{equation}

Let us come back to the equality 
\begin{equation} \label{c13}
  H_2^2(0) = \tilde{H}_2^2(0) \quad \Longleftrightarrow \quad \frac{(c^4 \det S)(0)}{s^{21}(0)} = \frac{(\tilde{c}^4 \det \tilde{S})(0)}{\tilde{s}^{21}(0)}.
\end{equation}
From (\ref{c3}) we thus obtain
\begin{equation} \label{c9}
  (c^4 \det S)(0) = C_0 \, (\tilde{c}^4 \det \tilde{S})(0).
\end{equation}
Likewise by working at $x^1 = A$, we obtain similarly
\begin{equation} \label{c9A}
  (c^4 \det S)(A) = C_1 \, (\tilde{c}^4 \det \tilde{S})(\tilde{A}).
\end{equation}

\begin{rem} \label{C0}
  Witout loss of generality, we can assume that the constants $C_0$ and $C_1$ appearing in (\ref{c3}) and (\ref{c3A}) are equal to $1$ in the following way. Recall that the DN maps are invariant by pullback by a diffeomorphim $\phi$ that is the identity on the boundary. Among such diffeomorphisms, we can use a change of coordinates of the form (\ref{CV})
$$
  y^1 = \int_0^{x^1} \sqrt{f_1}(s) ds \ \in [0,A_1], 	
$$
which leads to a new but equivalent expression of the metric $G$ given by (\ref{NewG}) - (\ref{NewS}) without changing the DN map $\Lambda_G$. In particular, in this new coordinate system, we can replace the initial first line of the St\"ackel matrix 
$$
  (s_{11}, s_{12}, s_{13}),
$$
by the new first line
$$
  (\frac{s_{11}}{f_1}, \frac{s_{12}}{f_1}, \frac{s_{13}}{f_1}).
$$
We see that we can always choose $f_1$ such that
\begin{equation} \label{Condf1}
  f_1(0) = C_0, \quad f_1(A) = C_1. 
\end{equation}
Putting this into (\ref{c3}) and (\ref{c3A}), we immediately see that, in this new coordinate system, the constants $C_0$ and $C_1$ disappear from (\ref{c3}) and (\ref{c3A}). Equivalently, this means that we can always assume from the very beginning that 
\begin{equation} \label{C0C1}
  C_0 = C_1 = 1.
\end{equation}
Note that the equalities (\ref{c3A}) and (\ref{c9A}) become in this new coordinate system
\begin{equation} \label{c10A}
  s^{31}(A_1) = \tilde{s}^{31}(\tilde{A}), \quad s^{21}(A_1) = \tilde{s}^{21}(\tilde{A}), \quad (c^4 \det S)(A_1) = (\tilde{c}^4 \det \tilde{S})(\tilde{A}),
\end{equation}
since the boundary $\{x^1 = A\}$ becomes $\{y^1 = A_1\}$. To avoid confusion, we identify $A$ with $A_1$ in what follows, that is we assume from the very beginning that we work with the variable $x^1 = y^1$. 
\end{rem}

Let us finish with the boundary determination results coming from (\ref{c0}). Recalling that the $R$-factor is given by (\ref{R}), \textit{i.e.}
$$
  R = \left( \frac{s^{11} s^{21} s^{31}}{c^4 \det S} \right)^{\frac{1}{4}}, 
$$
we see from (\ref{c3}), (\ref{c8}), (\ref{c9}), (\ref{C0C1}) and (\ref{c10A}) that
\begin{equation} \label{c11}
  R(0) = \tilde{R}(0), \quad R(A) = \tilde{R}(\tilde{A}).  
\end{equation}

Finally, having in mind the definition of the diagonal component $H_1^2$ of the metric $G$
$$
  H_1^2 = \frac{c^4 \det S}{s^{11}},
$$
we deduce from (\ref{c8}), (\ref{c9}), (\ref{C0C1}) and (\ref{c10A}) that
\begin{equation} \label{c12}
  H_1(0) = \tilde{H}_1(0), \quad H_1(A) = \tilde{H}_1(\tilde{A}).  
\end{equation}

\vspace{0.5cm}	
\noindent \textit{Step 2}. Assume next that the normal derivatives of the metrics $G$ and $\tilde{G}$ coincide on the boundary $\partial \Omega$, \textit{i.e.}
\begin{equation} \label{d0}
  \partial_\nu G = \partial_{\tilde{\nu}} \tilde{G}, \quad \textrm{on} \ \partial \Omega.
\end{equation}
At the boundary $\Omega_0 = \{x^1 = 0\}$, we observe that
$$
  \partial_\nu G_{|\Omega_0} = -\frac{1}{H_1(0)} \left( (\partial_1 H_2^2)(0) (dx^2)^2 + (\partial_1 H_3^2)(0) (dx^3)^2 \right). 
$$  
Hence from (\ref{d0}) we get 
$$
  \frac{(\partial_1 H_j^2)(0)}{H_1(0)} = \frac{(\partial_1 \tilde{H}_j^2)(0)}{\tilde{H}_1(0)}, \quad j=2,3, 
$$
which can be rewritten using (\ref{c1}) and (\ref{c12}) as
\begin{equation} \label{d1}
  (\partial_1 \log H_j^2)(0) = (\partial_1 \log \tilde{H}_j^2)(0), \quad j=2,3.
\end{equation}
Recalling that
$$
  H_j^2 = \frac{c^4 \det S}{s^{j1}}, \quad j=2,3,
$$
we infer from (\ref{d1}) that
\begin{equation} \label{d2}
  \left\{ \begin{array}{rcl} 
	(\partial_1 \log c^4 \det S)(0) - (\partial_1 \log s^{21})(0) & = & (\partial_1 \log \tilde{c}^4 \det \tilde{S})(0) - (\partial_1 \log \tilde{s}^{21})(0), \\
	(\partial_1 \log c^4 \det S)(0) - (\partial_1 \log s^{31})(0) & = & (\partial_1 \log \tilde{c}^4 \det \tilde{S})(0) - (\partial_1 \log \tilde{s}^{31})(0). 
	\end{array} \right. 
\end{equation}
Note here that we can use the same change of variables of the form (\ref{CV}) as in Remark \ref{C0} to assume from the very beginning that
\begin{equation} \label{d3}
  (\partial_1 \log s^{21})(0) = (\partial_1 \log \tilde{s}^{21})(0). 
\end{equation}
Indeed, it suffices to choose $f_1$ such that (\ref{Condf1}) and the additional condition
\begin{equation} \label{Condf1bis}
  (\partial_1 \log s^{21})(0) - (\partial_1 \log f_1)(0) = (\partial_1 \log \tilde{s}^{21})(0),  
\end{equation}
hold. Hence we obtain from (\ref{d2}) and (\ref{d3})
\begin{equation} \label{d4}
  \left\{ \begin{array}{rcl}
	(\partial_1 \log s^{21})(0) & = & (\partial_1 \log \tilde{s}^{21})(0), \\
	(\partial_1 \log s^{31})(0) & = & (\partial_1 \log \tilde{s}^{31})(0), \\
	(\partial_1 \log c^4 \det S)(0) & = & (\partial_1 \log \tilde{c}^4 \det \tilde{S})(0). 
	\end{array} \right.
\end{equation}
Of course, a similar analysis can be performed at the boundary component given by $\partial\Omega_1 =\{x^1 = A\}$. Thus we also obtain
\begin{equation} \label{d5}
  \left\{ \begin{array}{rcl}
	(\partial_1 \log s^{21})(A) & = & (\partial_1 \log \tilde{s}^{21})(\tilde{A}), \\
	(\partial_1 \log s^{31})(A) & = & (\partial_1 \log \tilde{s}^{31})(\tilde{A}), \\
	(\partial_1 \log c^4 \det S)(A) & = & (\partial_1 \log \tilde{c}^4 \det \tilde{S})(\tilde{A}). 
	\end{array} \right.
\end{equation}

Finally, recalling the definition of the contracted Christoffel symbol
$$
  \Gamma_1 = - \partial_1 \log \frac{H_2 H_3}{H_1} = -\frac{1}{2} \partial_1 \log \left( \frac{ c^4 \det S \,s^{11}}{s^{21} s^{31}} \right)= -\frac{1}{2} \partial_1 \log \left( \frac{ c^4 \det S}{s^{21} s^{31}} \right),
$$
we immediately obtain from (\ref{d4}) and (\ref{d5})
\begin{equation} \label{d6}
  \Gamma_1(0) = \tilde{\Gamma}_1(0), \quad \Gamma_1(A) = \tilde{\Gamma}_1(\tilde{A}). 
\end{equation}

\begin{rem} 
  The previous boundary determination results have been obtained using the variable $y^1$ defined by (\ref{CV}) such that (\ref{Condf1}) and (\ref{Condf1bis}) hold. Observe that, putting (\ref{C0C1}) into (\ref{c4}) and using (\ref{c8}), the use of the variable $y^1$ with the identification $A_1 = A$ implies the following equalities
\begin{equation} \label{s12}
  s_{12}(0) = \tilde{s}_{12}(0), \quad s_{13}(0) = \tilde{s}_{13}(0), \quad s_{12}(A) = \tilde{s}_{12}(\tilde{A}), \quad s_{13}(A) = \tilde{s}_{13}(\tilde{A}). 
\end{equation} 
Similarly, the two first lines of (\ref{d4}) together with (\ref{c3}) can be rewritten as
$$
  T \left( \begin{array}{c} -s'_{13}(0) \\ s'_{12}(0) \end{array} \right) = \tilde{T} \left( \begin{array}{c} -\tilde{s}'_{13}(0) \\ \tilde{s}'_{12}(0) \end{array} \right).  
$$
Hence, using (\ref{c8}) once more and a similar analysis at $y^1 = A$, we obtain the equalities 
\begin{equation} \label{s12'}
  s'_{12}(0) = \tilde{s}'_{12}(0), \quad s'_{13}(0) = \tilde{s}'_{13}(0), \quad s'_{12}(A) = \tilde{s}'_{12}(\tilde{A}), \quad s'_{13}(A) = \tilde{s}'_{13}(\tilde{A}). 
\end{equation}
Notice that the equalities (\ref{s12}) and (\ref{s12'}) could be used to define the change of variable $y^1$ from the outset.  
\end{rem}

\vspace{0.5cm}
\noindent \textit{Step 3}. The previous results obtained in Steps 1 and 2 exhaust all the information that we can extract from boundary determination arguments. To go further, we need to exploit the particular structure (\ref{DNStructure}) - (\ref{AG}) of the DN map. 

Recall first that
$$
  \Lambda_G = \left( \begin{array}{cc} \frac{-1}{H_1(0)} & 0 \\ 0 & \frac{1}{H_1(A)} \end{array} \right) \Bigg[ 
	            \left( \begin{array}{cc} \frac{\Gamma_1(0)}{2} & 0 \\ 0 & \frac{\Gamma_1(A)}{2} \end{array} \right) + \ \left( \begin{array}{cc} R(0) & 0 \\ 0 & R(A) \end{array} \right) A_G \left( \begin{array}{cc} \frac{1}{R(0)} & 0 \\ 0 & \frac{1}{R(A)} \end{array} \right) \Bigg]
$$
where 
$$
  A_G = \bigoplus_{m=1}^\infty A_G^m, \quad A_G^m := (A_G)_{|\langle Y_m \rangle} = \left( \begin{array}{cc} M(\mu_m^2,\nu_m^2)  & \frac{1}{\Delta(\mu_m^2,\nu_m^2)} \\ \frac{1}{\Delta(\mu_m^2,\nu_m^2)} & N(\mu_m^2,\nu_m^2) \end{array} \right).
$$
From $\Lambda_G = \Lambda_{\tilde{G}}$ and from the boundary determination results (\ref{c11}), (\ref{c12}) and (\ref{d6}), we obtain immediately the equality between operators 
\begin{equation} \label{e0}
  A_G = A_{\tilde{G}}, \quad \textrm{on} \ \ L^2(\mathbb{T}^2; \, s^{11} dx^2 dx^3) \otimes L^2(\mathbb{T}^2; \, s^{11} dx^2 dx^3). 
\end{equation}
Denote by $(\omega_m, \varphi_m)_{m \in \Z^*}$ and $(\tilde{\omega}_m, \tilde{\varphi}_m)_{m \in \Z^*}$ the eigenvalues and eigenfunctions of $A_G$ and $A_{\tilde{G}}$ respectively. As a consequence of (\ref{e0}) we have \footnote{In fact, the common eigenvalues and corresponding eigenspaces $(\omega_m, E_m)$ of $A_G$ and $A_{\tilde{G}}$ coincide. But remember that the operator $A_G$ does not depend on the choice of a Hilbert basis within an eigenspace $E_m$ so that without loss of generality we can also identify the eigenfunctions. }  
\begin{equation} \label{e1}
  \omega_m = \tilde{\omega}_m, \quad \varphi_m = \tilde{\varphi}_m, \quad \forall m \in \Z^*. 
\end{equation}

Let us make explicit the eigenvalues $(\omega_m)_{m \in \Z^*}$ and their corresponding eigenfunctions $(\varphi_m)_{m \in \Z^*}$. Recall that on each harmonic $\langle Y_m \rangle, \ m \geq 1$, the operator $A_G$ simplifies in the $2\times2$-matrix given by (\ref{AG}), \textit{i.e.}   
$$
  A_G^m = \left( \begin{array}{cc} M(\mu_m^2,\nu_m^2)  & \frac{1}{\Delta(\mu_m^2,\nu_m^2)} \\ \frac{1}{\Delta(\mu_m^2,\nu_m^2)} & N(\mu_m^2,\nu_m^2) \end{array} \right),
$$
Diagonalizing, we obtain two eigenvalues attached to each index $m \geq 1$
\begin{equation} \label{e2}
	  \omega_{\pm m} := \frac{M_m + N_m \pm \sqrt{(M_m - N_m)^2 + \frac{4}{\Delta_m^2}}}{2}, \quad m \geq 1, 
\end{equation}
associated to the eigenfunctions
\begin{equation} \label{e3}
  \varphi_m^+ = \left( \begin{array}{c} 1 \\ x^+_m \end{array} \right) \otimes Y_m, \quad \varphi_m^- = \left( \begin{array}{c} 1 \\ x^-_m \end{array} \right) \otimes Y_m, \quad m \geq 1, 
\end{equation}
where we used the shorthand notation
$$
  \Delta_m = \Delta(\mu_m^2,\nu_m^2), \quad M_m = M(\mu_m^2,\nu_m^2), \quad N_m = N(\mu_m^2,\nu_m^2), 
$$
and
$$
  x^\pm_m = \frac{\Delta_m}{2} \left( N_m - M_m \pm \sqrt{(M_m+N_m)^2 + \frac{4}{\Delta_m^2}} \right).
$$
In particular, using (\ref{e1}), (\ref{e2}) and (\ref{e3}), we see that there exists a bijection $\theta: \, \N^* \longrightarrow \N^*$ such that
\begin{equation} \label{e4}
  Y_m = \tilde{Y}_{\theta(m)}, \quad \forall m \geq 1. 
\end{equation}
It is not clear at this stage why the bijection $\theta$ should be the identity. However, we can proceed in the following way. Recall from (\ref{HL}) and (\ref{A}) that
\begin{equation} \label{e5}
  \left( \begin{array}{c} -\partial_2^2 \\ -\partial_3^2 \end{array} \right) = - T \left( \begin{array}{c} H \\ L \end{array} \right) + \left( \begin{array}{c} \phi_2 \\ \phi_3 \end{array} \right) = - \tilde{T} \left( \begin{array}{c} \tilde{H} \\ \tilde{L} \end{array} \right) + \left( \begin{array}{c} \tilde{\phi}_2 \\ \tilde{\phi}_3 \end{array} \right),
\end{equation}
where $T$ denotes the $2\times2$-matrix (\ref{T}) and $T = \tilde{T}$ thanks to (\ref{c8}). If we apply the operators in the equality (\ref{e5}) to $Y_m = \tilde{Y}_{\theta(m)}$, we thus obtain
\begin{equation} \label{e6}
  \left[ - T \left( \begin{array}{c} \mu_m^2 \\ \nu_m^2 \end{array} \right) + \left( \begin{array}{c} \phi_2 \\ \phi_3 \end{array} \right) \right] \otimes Y_m = \left[ - T \left( \begin{array}{c} \tilde{\mu}_{\theta(m)}^2 \\ \tilde{\nu}_{\theta(m)}^2 \end{array} \right) + \left( \begin{array}{c} \tilde{\phi}_2 \\ \tilde{\phi}_3 \end{array} \right) \right] \otimes Y_m.
\end{equation}
If we divide (\ref{e6}) by $Y_m$ (outside the nodal sets), we get using the invertibility of $T$ that 
\begin{equation} \label{e7}
  T^{-1} \left( \begin{array}{c} \tilde{\phi}_2 - \phi_2 \\ \tilde{\phi}_3 - \phi_3 \end{array} \right) = \left( \begin{array}{c} \tilde{\mu}_{\theta(m)}^2 - \mu_m^2 \\ \tilde{\nu}_{\theta(m)}^2 - \nu_m^2 \end{array} \right). 
\end{equation}
The remarkable properties of St\"ackel metrics manifest themselves again here since the left-hand side of (\ref{e7}) depends on the variables $(x^2,x^3)$ while the right-hand side of (\ref{e7}) is given by constants. Hence we infer that there exist two constants $B_1$ and $B_2$ such that
\begin{equation} \label{e8}
  T^{-1} \left( \begin{array}{c} \tilde{\phi}_2 - \phi_2 \\ \tilde{\phi}_3 - \phi_3 \end{array} \right) = \left( \begin{array}{c} B_1 \\ B_2 \end{array} \right) = \left( \begin{array}{c} \tilde{\mu}_{\theta(m)}^2 - \mu_m^2 \\ \tilde{\nu}_{\theta(m)}^2 - \nu_m^2 \end{array} \right). 
\end{equation}
Putting the first equality of (\ref{e8}) into (\ref{e5}), we get easily
\begin{equation} \label{e9}
  \left( \begin{array}{c} \tilde{H} \\ \tilde{L} \end{array} \right) = \left( \begin{array}{c} H \\ L \end{array} \right) + \left( \begin{array}{c} B_1 \\ B_2 \end{array} \right), 
\end{equation}
that is the angular operators $(H,L)$ differ from $(\tilde{H}, \tilde{L})$ by mere constants $B_1$  and $B_2$. In particular, their common eigenfunctions are the same
\begin{equation} \label{e10}
  Y_m = \tilde{Y}_m, \quad \forall m \geq 1, 
\end{equation}
and their corresponding eigenvalues differ by the same constants
$$
  \tilde{\mu}_m^2= \mu_m^2 + B_1, \quad \tilde{\nu}_m^2 = \nu_m^2 + B_2, \quad \forall m \geq 1.  
$$
Finally, as pointed out in Section \ref{II2}, after Proposition \ref{StackelForm}, we can set the constants $B_1$ and $B_2$ equal to zero since this corresponds to an intrinsic invariance of the separation of variables' procedure. We thus assume from now on that
\begin{equation} \label{e11}
  \mu_m^2 = \tilde{\mu}_m^2, \quad \nu_m^2 = \tilde{\nu}_m^2, \quad \forall m \geq 1,  
\end{equation}
and (from (\ref{e8}))
\begin{equation} \label{e12}
  \phi_2 = \tilde{\phi}_2, \quad \phi_3 = \tilde{\phi}_3. 
\end{equation}

\vspace{0.5cm}
\noindent \textit{Summary}. At this stage of the proof, under our main assumption $\Lambda_G = \Lambda_{\tilde{G}}$, we have shown
$$
  \left( \begin{array}{cc} s_{22} & s_{23} \\ s_{32} & s_{33} \end{array} \right) = \left( \begin{array}{cc} \tilde{s}_{22} & \tilde{s}_{23} \\ \tilde{s}_{32} & \tilde{s}_{33} \end{array} \right), \quad \phi_2 = \tilde{\phi}_2, \quad \phi_3 = \tilde{\phi}_3,
$$
and
$$
  \mu_m^2 = \tilde{\mu}_m^2, \quad \nu_m^2 = \tilde{\nu}_m^2, \quad Y_m = \tilde{Y}_m, \quad \forall m \geq 1.
$$

% CAM METHOD

\subsection{The multi-parameter CAM method} \label{III3}

We continue extracting more informations on the metrics $G$ and $\tilde{G}$ from the equality between operators
$$
  A_G = A_{\tilde{G}}, 
$$
and the coincidence between the angular parts of the metrics $G$ and $\tilde{G}$ that leads to the equalities
$$	
	\mu_m^2 = \tilde{\mu}_m^2, \quad \nu_m^2 = \tilde{\nu}_m^2, \quad Y_m = \tilde{Y}_m, \quad \forall m \geq 1,
$$
proved in Section \ref{III2}. Using the definition of $A_G$, this implies in particular that
\begin{equation} \label{f1}
  M(\mu_m^2, \nu_m^2) = \tilde{M}(\mu_m^2, \nu_m^2), \quad \forall m \geq 1. 
\end{equation}
Hence the two WT functions associated to the separated ODE
\begin{eqnarray}
	-u_m'' + [\,\mu_m^2 s_{12}(x^1) + \nu_m^2 s_{13}(x^1) - \phi_1(x^1)] \, u_m & = & 0, \quad \quad x^1 \in (0,A), \label{Sepa}\\ 
	-\tilde{u}_m'' + [\,\mu_m^2 \tilde{s}_{12}(x^1) + \nu_m^2 \tilde{s}_{13}(x^1) - \tilde{\phi}_1(x^1)] \, \tilde{u}_m & = & 0, \quad \quad x^1 \in (0,\tilde{A}), \label{Sepb} 
\end{eqnarray}
coincide when evaluated on the joint spectrum $J := \{(\mu_m^2, \nu_m^2), \ m \geq 1 \}$. 

Our first task will be to show that the equality (\ref{f1}) on the discrete subset $J$ can be extended to the whole plane $\C^2$, \textit{i.e.}
\begin{equation} \label{f2}
  M(\mu^2, \nu^2) = \tilde{M}(\mu^2, \nu^2), \quad \forall (\mu, \nu) \in \C^2 \backslash \{poles\} 
\end{equation}
The passage from (\ref{f1}) to (\ref{f2}) is what we call the multi-parameter CAM method and turns out to be the central technical tool from which we will able to solve the inverse problem. To do this, note first that (\ref{f1}) can be rewritten using (\ref{WT}) as 
\begin{equation} \label{f3}
  D(\mu_m^2, \nu_m^2) \tilde{\Delta}(\mu_m^2, \nu_m^2) - \tilde{D}(\mu_m^2, \nu_m^2) \Delta(\mu_m^2, \nu_m^2) = 0, \quad \forall m \geq 1. 
\end{equation}
Define now the function
\begin{equation} \label{F}
  F(\mu,\nu) := D(\mu^2, \nu^2) \tilde{\Delta}(\mu^2, \nu^2) - \tilde{D}(\mu^2, \nu^2) \Delta(\mu^2, \nu^2).
\end{equation}
Then $F$ is clearly analytic \footnote{The functions $c_j, \, s_j, \ j=0,1$ and thus the functions $\Delta$, $D$ and $F$ are analytic in the variables $\mu$ and $\nu$ independently thanks to standard theorems on ODE depending analytically on parameters. Hence the function $F$ is analytic on $\C^2$ due to the Hartogs Theorem.} on $\C^2$ and vanishes on the "square-root" of the joint spectrum $J$ thanks to (\ref{f3}). Hence, in order to prove (\ref{f2}), it will be enough to prove that $F$ vanishes identically. 

To go further, we will use the following result of Berndtsson \cite{Ber1978} (that we took from Bloom \cite{Bl1990}) which provides a sufficient condition for a discrete set to be a \emph{uniqueness set} of a bounded analytic function of several variables. 

\begin{thm}[Berndtsson, 1978] \label{Be}
Let $K$ be an open cone in $\R^n$ with vertex at the origin and $T(K) = \{ z \in \C^n \ / \ \Re(z) \in K \}$. Suppose $f$ is bounded and analytic on $T(K)$. Let $E$ be a discrete subset of $K$ such that for some constant $h>0$, $e_1, \, e_2 \in E$ implies that $|e_1 - e_2| \geq h$. Let $n(r) = \# E \cap B(0,r)$. Assume that $f$ vanishes on $E$. Then $f$ is identically $0$ if 
$$
  \varlimsup_{r \to \infty} \frac{n(r)}{r^n} > 0. 
$$ 
\end{thm}

In order to apply Theorem \ref{Be}, we need to define an analytic function that is bounded on a conic set of the form $T(K)$ and that satisfies the above properties. The natural candidate - the function $F$ - is not bounded and we need to rescale it in a convenient way. Hence let us first prove some universal estimates for $F$. 

\begin{prop} \label{EstF}
There exist positive constants $\bar{A}, \bar{B}, C > 0$ such that for all $(\mu,\nu) \in \C^2$
$$
  |D(\mu^2, \nu^2)|, \, |\Delta(\mu^2, \nu^2)| \ \leq \ C \, e^{\frac{\bar{A}}{2} |\Re(\mu)| + \frac{\bar{B}}{2} |\Re(\nu)|}. 
$$   
As a consequence,
$$
  |F(\mu, \nu)| \ \leq \ C \, e^{\bar{A} |\Re(\mu)| + \bar{B} |\Re(\nu)|}.
$$
\end{prop}

The proof of this Proposition requires some preliminary steps. \\

\vspace{0.5cm}
\noindent \textit{Step 1}. We first prove some estimates in $\mu$ and $\nu$ independently. 

\begin{lemma} \label{Estimatemunu}
1. For each $\nu \in \C$ fixed, there exists positive constants $\bar{A}, \, C(\nu) > 0$ such that
$$
  |D(\mu^2, \nu^2)|, \, |\Delta(\mu^2, \nu^2)|, \, |F(\mu, \nu)| \ \leq \ C(\nu) \, e^{\bar{A} |\Re(\mu)|}.
$$ 
2. For each $\mu \in \C$ fixed, there exists positive constants $\bar{B}, \, C(\nu) > 0$ such that
$$
  |D(\mu^2, \nu^2)|, \, |\Delta(\mu^2, \nu^2)|, \, |F(\mu, \nu)| \ \leq \ C(\mu) \, e^{\bar{B} |\Re(\nu)|}.
$$
\end{lemma}

In order to prove this Lemma, we need to recast the separated ODE (\ref{Sepa}) and (\ref{Sepb}) into \emph{normal forms}, that is Schr\"odinger equations with spectral parameters $-\mu^2$ or $-\nu^2$. For instance, if we choose $-\mu^2$  as spectral parameter, we introduce the new radial coordinate
\begin{equation} \label{u1} 
  u^1 = \int_0^{x^1} \sqrt{s_{12}(t)} dt \ \ \in \ \ [0,\bar{A}],
\end{equation}
and remark that, if $u(x^1,\mu^2,\nu^2)$ is a solution of the separated ODE
$$
  -u'' + [\,\mu^2 s_{12}(x^1) + \nu^2 s_{13}(x^1) - \phi_1(x^1)] \, u = 0,
$$
then the function $U(u^1,\mu^2,\nu^2) := (s_{12}(x^1(u^1)))^{\frac{1}{4}} u(x^1(u^1), \mu^2, \nu^2)$ is a solution of the ODE
\begin{equation} \label{EqU}
  -\ddot{U} + q_\nu(u^1) U = -\mu^2 U,
\end{equation}
where
\begin{equation} \label{qnu-psi1}
  q_\nu = \nu^2 \bar{s}_{13} - \bar{\phi}_1, \quad \bar{s}_{13} = \frac{s_{13}}{s_{12}}, \quad \bar{\phi}_1  = \frac{\phi_1}{s_{12}} - \frac{(\dot{\log s_{12}})^2}{16} + \frac{(\ddot{\log s_{12}})}{4}.
\end{equation}
Here, the notation $\,\dot{}\,$ denotes the derivative with respect to $u^1$. For fixed $\nu \in \C$, (\ref{EqU}) is now a classical Schr\"odinger equation with $-\mu^2$ as spectral parameter. 

Introduce the FSS $\{U_0,V_0\}$ and $\{U_1,V_1\}$ of (\ref{EqU}) defined by the Cauchy conditions
\begin{eqnarray*}
  U_0(0) = 1, \ \dot{U}_0(0) = 0, \quad V_0(0) = 0, \ \dot{V}_0(0) = 1, \\
	U_1(\bar{A}) = 1, \ \dot{U}_1(\bar{A}) = 0, \quad V_1(\bar{A}) = 0, \ \dot{V}_1(\bar{A}) = 1, 
\end{eqnarray*}
as well as the characteristic and WT functions 
\begin{equation} \label{NewWT}
  \Delta_{q_\nu}(\mu^2) = W(V_0,V_1), \quad D_{q_\nu}(\mu^2) = W(U_0,V_1), \quad M_{q_\nu}(\mu^2) = -\frac{D_{q_\nu}(\mu^2)}{\Delta_{q_\nu}(\mu^2)},
\end{equation}
where $W(f,g) := f \dot{g} - \dot{f} g$ is the Wronskian. Introduce also the FSS $\{C_0,S_0\}$ and $\{C_1,C_1\}$ of (\ref{EqU}) defined by
\begin{equation} \label{CS}
  \begin{array}{cc} C_j(u^1,\mu^2,\nu^2) : = (s_{12}(x^1(u^1)))^{\frac{1}{4}} c_j(x^1(u^1), \mu^2,\nu^2), & j=0,1, \\
	  S_j(u^1,\mu^2,\nu^2) : = (s_{12}(x^1(u^1)))^{\frac{1}{4}} s_j(x^1(u^1), \mu^2,\nu^2), & j=0,1.
		\end{array}
\end{equation}
Then a straighforward though tedious calculation shows that
\begin{equation} \label{UCVS0}
  C_0 = (s_{12}(0))^{\frac{1}{4}} U_0 + \left( \frac{s'_{12}(0)}{4 (s_{12}(0))^{\frac{5}{4}}} \right) V_0, \quad S_0 = \frac{V_0}{(s_{12}(0))^{\frac{1}{4}}}, 
\end{equation}
and
\begin{equation} \label{UCVS1}
  C_1 = (s_{12}(A))^{\frac{1}{4}} U_1 + \left( \frac{s'_{12}(A)}{4 (s_{12}(A))^{\frac{5}{4}}} \right) V_1, \quad S_1 = \frac{V_A}{(s_{12}(A))^{\frac{1}{4}}}.  
\end{equation}
From (\ref{UCVS0}) and (\ref{UCVS1}), we also obtain
\begin{equation} \label{LinkDelta}
  \Delta(\mu^2,\nu^2) = \frac{1}{(s_{12}(0)\,s_{12}(A))^{\frac{1}{4}}} \Delta_{q_\nu}(\mu^2), 
\end{equation}
\begin{equation} \label{LinkD}	
	D(\mu^2,\nu^2) = \left( \frac{s_{12}(0)}{s_{12}(A)} \right)^{\frac{1}{4}} D_{q_\nu}(\mu^2) + \frac{s'_{12}(0)}{4 (s_{12}(0))^{\frac{5}{4}} (s_{12}(A))^{\frac{1}{4}}} \Delta_{q_\nu}(\mu^2),
\end{equation}
and 
\begin{equation} \label{LinkM}
  M(\mu^2,\nu^2) = \frac{1}{4} (\log s_{12})'(0) + \sqrt{s_{12}(0)} \,M_{q_\nu}(\mu^2).  
\end{equation}

The interest in introducing such a normal form comes from the existence of universal asymptotics as $|\mu| \to \infty$ for the functions $U_j, \, V_j, \, j=0,1$ (see \cite{PT1987}, Theorem 3, p.13). Precisely we have for fixed $\nu \in \C$ 
\begin{equation} \label{UnivEst0}
  \left\{ \begin{array}{rcl} 
	U_0(u^1,\mu^2,\nu^2) & = & \cosh(\mu u^1) + O\left( \frac{1}{|\mu|} e^{|\Re(\mu)|u^1 + \| q_\nu \| \sqrt{u^1}} \right), \\
	\dot{U}_0(u^1,\mu^2,\nu^2) & = & -\mu \sinh(\mu u^1) + O\left( \| q_\nu \| e^{|\Re(\mu)|u^1 + \| q_\nu \| \sqrt{u^1}} \right), \\
	V_0(u^1,\mu^2,\nu^2) & = & \frac{\sinh(\mu u^1)}{\mu} + O\left( \frac{1}{|\mu|^2} e^{|\Re(\mu)|u^1 + \| q_\nu \| \sqrt{u^1}} \right), \\
  \dot{V}_0(u^1,\mu^2,\nu^2) & = & \cosh(\mu u^1) + O\left( \frac{\| q_\nu \|}{|\mu|} e^{|\Re(\mu)|u^1 + \| q_\nu \| \sqrt{u^1}} \right),
	\end{array} \right. \quad |\mu| \to \infty,
\end{equation}
and 
\begin{equation} \label{UnivEst1}
  \left\{ \begin{array}{rcl} 
	U_1(u^1,\mu^2,\nu^2) & = & \cosh(\mu (\bar{A}-u^1)) + O\left( \frac{1}{|\mu|} e^{|\Re(\mu)|(\bar{A}-u^1) + \| q_\nu \| \sqrt{\bar{A}-u^1}} \right), \\
	\dot{U}_1(u^1,\mu^2,\nu^2) & = & -\mu \sinh(\mu (\bar{A}-u^1)) + O\left( \| q_\nu \| e^{|\Re(\mu)| (\bar{A}-u^1) + \| q_\nu \| \sqrt{\bar{A}-u^1}} \right), \\
	V_1(u^1,\mu^2,\nu^2) & = & -\frac{\sinh(\mu (\bar{A}-u^1))}{\mu} + O\left( \frac{1}{|\mu|^2} e^{|\Re(\mu)| (\bar{A}-u^1) + \| q_\nu \| \sqrt{\bar{A}-u^1}} \right), \\
  \dot{V}_1(u^1,\mu^2,\nu^2) & = & \cosh(\mu (\bar{A}-u^1)) + O\left( \frac{\| q_\nu \|}{|\mu|} e^{|\Re(\mu)| (\bar{A}-u^1) + \| q_\nu \| \sqrt{\bar{A}-u^1}} \right),
	\end{array} \right. \quad |\mu| \to \infty.
\end{equation}
From these universal estimates and the definition of $\Delta_{q_\nu}(\mu^2), \, D_{q_\nu}(\mu^2)$ given by (\ref{NewWT}), we obtain that for each fixed $\nu \in \C$, there exists a constant $C(\nu)$ such that 
\begin{equation} \label{Deltamu}
  \Delta_{q_\nu}(\mu^2) = \frac{\sinh(\bar{A} \mu )}{\mu} + O\left( C(\nu) \frac{e^{|\Re(\mu)| \bar{A}}}{|\mu|^2} \right), \quad |\mu| \to \infty,
\end{equation}
and
\begin{equation} \label{Dmu}
  D_{q_\nu}(\mu^2) = \cosh(\bar{A} \mu ) + O\left( C(\nu) \frac{e^{|\Re(\mu)| \bar{A}}}{|\mu|} \right), \quad |\mu| \to \infty.
\end{equation}
We can now give the proof of Lemma \ref{Estimatemunu}. 

\begin{proof}[Proof of Lemma \ref{Estimatemunu}]
The proof of 1. follows directly from the estimates (\ref{Deltamu}) and (\ref{Dmu}) together with (\ref{F}), (\ref{LinkDelta}) and (\ref{LinkD}). The proof of 2. is similar to 1. inverting the role \footnote{For this, it suffices to use the new coordinate $$v^1 = \ds\int_0^{x^1} \sqrt{s_{13}(t)} dt \ \ \in \ \ [0,\bar{B}],$$ and remark that, if $u(x^1,\mu^2,\nu^2)$ is a solution of the separated ODE $-u'' + [\,\mu^2 s_{12}(x^1) + \nu^2 s_{13}(x^1) - \phi_1(x^1)] \, u = 0$, then the function $V(v^1,\mu^2,\nu^2) := (s_{13}(x^1(v^1)))^{\frac{1}{4}} u(x^1(v^1), \mu^2, \nu^2)$ is a solution of the ODE
$$
  -\ddot{V} + q_\mu(v^1) V = -\nu^2 V, \quad q_\mu = \mu^2 \frac{s_{12}}{s_{13}} - \frac{\phi_1}{s_{13}} + \frac{(\dot{\log s_{13}})^2}{16} - \frac{(\ddot{\log s_{13}})}{4}.
$$}
of the spectral parameters $\mu^2$ and $\nu^2$. We leave the details to the readers. 
\end{proof}

\vspace{0.5cm}
\noindent \textit{Step 2}. Second we need a uniform estimate when $(\mu, \nu) = (iy,iy') \in (i\R)^2$. 

\begin{lemma} \label{EstimateiR}
There exists a constant $C>0$ such that for all $(y, y') \in \R^2$
$$
  |D(-y^2, -y'^2)|, \, |\Delta(-y^2, -y'^2)|, \,|F(iy,iy')| \ \leq \ C. 
$$	
\end{lemma}

\begin{proof}
When $(\mu, \nu) = (iy,iy') \in (i\R)^2$, the separated ODE (\ref{Sepa}) takes the form
\begin{equation} \label{Sepc}
  -u'' - \phi_1(x^1) u = \omega^2 r_{yy'}(x^1) u,  
\end{equation}
where
\begin{equation} \label{omegar}
  \omega^2 = y^2 + y'^2, \quad r_{yy'}(x^1) = \frac{y^2 s_{12}(x^1)  +y'^2 s_{13}(x^1)}{y^2 + y'^2}. 
\end{equation}
We introduce the change of variable (which is dependent on $y, y'$!)
$$
  w^1 = w^1_{yy'} = \int_0^{x^1} \sqrt{r_{yy'}(t)} dt \ \ \in \ \ [0,\bar{C}_{yy'}],
$$ 
and remark that, if $u(x^1,\mu^2,\nu^2)$ is a solution of the separated ODE
$$
  -u'' + [\,\mu^2 s_{12}(x^1) + \nu^2 s_{13}(x^1) - \phi_1(x^1)] \, u = 0,
$$
then the function $W(w^1,-y^2,-y'^2) := (r_{yy'}(x^1(w^1)))^{\frac{1}{4}} W(x^1(w^1), -y^2, -y'^2)$ is a solution of the ODE
\begin{equation} \label{EqW}
  -\ddot{W} + q_{yy'}(w^1) W = \omega^2 W,
\end{equation}
where
\begin{equation} \label{qyy'}
  q_{yy'} = -\frac{\phi_1}{r_{yy'}} + \frac{(\dot{\log r_{yy'}})^2}{16} - \frac{(\ddot{\log r_{yy'}})}{4}.
\end{equation}
We also observe the following properties related to this change of variable. Since $s_{12}, \, s_{13}$ are positive and $C^\infty$ on $[0,A]$, we deduce first that there exist positive constants $c,C > 0$ such that
$$
  c \leq r_{yy'} \leq C, 
$$
and second that the potential $q_{yy'}$ is uniformly bounded with respect to $(y, y') \in \R^2$. Finally, the variable $w^1$ lives on the $(y,y')$-dependent interval $[0, \bar{C} = \bar{C}_{yy'}]$ whose length satisfies the uniform estimate
$$
  \sqrt{c} A \leq \bar{C}_{yy'} = \int_0^{A} \sqrt{r_{yy'}(t)} dt \leq \sqrt{C} A. 
$$
Similarly to what we did previously, we introduce the FSS $\{W_0,X_0\}$ and $\{W_1,X_1\}$ of (\ref{EqW}) defined by the Cauchy conditions
\begin{eqnarray*}
  W_0(0) = 1, \ \dot{W}_0(0) = 0, \quad X_0(0) = 0, \ \dot{X}_0(0) = 1, \\
	W_1(\bar{C}) = 1, \ \dot{W}_1(\bar{C}) = 0, \quad X_1(\bar{C}) = 0, \ \dot{X}_1(\bar{C}) = 1, 
\end{eqnarray*}
as well as the characteristic functions 
\begin{equation} \label{NewChar}
  \Delta_{q_{yy'}}(\omega^2) = W(X_0,X_1), \quad D_{q_{yy'}}(\omega^2) = W(W_0,X_1). 
\end{equation}
These new characteristic functions are related to the initial ones by the formulas
\begin{equation} \label{NewLinkDelta}
  \Delta(-y^2, -y'^2) = \frac{1}{(r_{yy'}(0)\,r_{yy'}(A))^{\frac{1}{4}}} \Delta_{q_{yy'}}(\omega^2), 
\end{equation}
\begin{equation} \label{NewLinkD}	
	D(-y^2, -y'^2) = \left( \frac{r_{yy'}(0)}{r_{yy'}(A)} \right)^{\frac{1}{4}} D_{q_{yy'}}(\omega^2) + \frac{r'_{yy'}(0)}{4 (r_{yy'}(0))^{\frac{5}{4}} (r_{yy'}(A))^{\frac{1}{4}}} \Delta_{q_{yy'}}(\omega^2).
\end{equation}
It is thus enough to show that the new characteristic functions are uniformly bounded on $(y, y') \in \R^2$. This can be done once again using the universal estimates from \cite{PT1987}. Precisely, since $\Im(\omega) = 0$, we have
\begin{equation} \label{NewUnivEst0}
  \left\{ \begin{array}{rcl} 
	W_0(w^1,-y^2,-y'^2) & = & \cos(\omega w^1) + O\left( \frac{1}{|\omega|} e^{\| q_{yy'} \| \sqrt{w^1}} \right), \\
	\dot{W}_0(w^1,-y^2,-y'^2) & = & - \omega \sin(\omega w^1) + O\left( \| q_{yy'} \| e^{\| q_{yy'} \| \sqrt{w^1}} \right), \\
	X_0(w^1,-y^2,-y'^2) & = & \frac{\sin(\omega w^1)}{\omega} + O\left( \frac{1}{|\omega|^2} e^{\| q_{yy'} \| \sqrt{w^1}} \right), \\
  \dot{X}_0(w^1,-y^2,-y'^2) & = & \cos(\omega w^1) + O\left( \frac{\| q_{yy'} \|}{|\omega|} e^{\| q_{yy'} \| \sqrt{w^1}} \right),
	\end{array} \right. \quad |\omega| \to \infty,
\end{equation}
and 
\begin{equation} \label{NewUnivEst1}
  \left\{ \begin{array}{rcl} 
	W_1(w^1,-y^2,-y'^2) & = & \cos(\omega (\bar{C}-w^1)) + O\left( \frac{1}{|\omega|} e^{\| q_{yy'} \| \sqrt{\bar{C}-w^1}} \right), \\
	\dot{W}_1(w^1,-y^2,-y'^2) & = & -\omega \sin(\omega (\bar{C}-w^1)) + O\left( \| q_{yy'} \| e^{\| q_{yy'} \| \sqrt{\bar{C}-w^1}} \right), \\
	X_1(w^1,-y^2,-y'^2) & = & -\frac{\sin(\omega (\bar{C}-w^1))}{\omega} + O\left( \frac{1}{|\omega|^2} e^{\| q_{yy'} \| \sqrt{\bar{C}-w^1}} \right), \\
  \dot{X}_1(w^1,-y^2,-y'^2) & = & \cos(\omega (\bar{C}-w^1)) + O\left( \frac{\| q_{yy'} \|}{|\omega|} e^{\| q_{yy'} \| \sqrt{\bar{C}-w^1}} \right),
	\end{array} \right. \quad |\omega| \to \infty.
\end{equation}
Therefore we obtain (since the potentials $q_{yy'}$ and $\bar{C}_{yy'}$ are uniformly bounded with respect to $(y, y') \in \R^2$)
\begin{equation} \label{Deltaomega}
  \Delta_{q_{yy'}}(\omega^2) = \frac{\sin(\bar{C}_{yy'} \omega )}{\omega} + O\left(\frac{1}{|\omega|^2} \right), \quad |\omega| \to \infty,
\end{equation}
and
\begin{equation} \label{Domega}
  D_{q_{yy'}}(\omega^2) = \cos(\bar{C}_{yy'} \omega ) + O\left( \frac{1}{|\omega|} \right), \quad |\omega| \to \infty.
\end{equation}
Finally, we deduce from (\ref{F}), (\ref{NewLinkDelta}), (\ref{NewLinkD}), (\ref{Deltaomega}) and (\ref{Domega}) that there exists a constant $C>0$ such that
\begin{equation} \label{f4}
  |F(iy,iy')| \leq \frac{C}{|\omega|}, \quad |\omega| \to \infty.
\end{equation}
The claim of the Lemma follows from (\ref{f4}). 

\end{proof}

\vspace{0.5cm}
We can now finish the proof of Proposition \ref{EstF} by applying twice the Phragmen-Lindel\"of principle. 

\begin{proof}[Proof of Proposition \ref{EstF}]
First we fix $\nu \in i\R$. According to Lemmas \ref{Estimatemunu} and \ref{EstimateiR}, the analytic function $\mu \longrightarrow F(\mu,\nu)$ satisfies
$$
  \left\{ \begin{array}{lc}
	|F(\mu,\nu)| \leq C(\nu) e^{\bar{A}|\Re(\mu)|}, & \forall \mu \in \C, \\
	|F(\mu,\nu)| \leq C, & \forall \mu \in i\R. 
	\end{array} \right. 
$$
Hence the Phragmen-Lindel\"of principle (see for instance \cite{Lev1996}, Lecture 6., Theorem 3) yields
\begin{equation} \label{f5}
  |F(\mu,\nu)| \leq C e^{\bar{A}|\Re(\mu)|}, \quad \forall (\mu, \nu) \in (\C,i\R). 
\end{equation}
Second we fix $\mu \in \C$. Then, according to Lemma \ref{Estimatemunu} and (\ref{f5}), the analytic function $\nu \longrightarrow F(\mu,\nu)$ satisfies
$$
  \left\{ \begin{array}{lc}
	|F(\mu,\nu)| \leq C(\mu) e^{\bar{B}|\Re(\nu)|}, & \forall \nu \in \C, \\
	|F(\mu,\nu)| \leq C e^{\bar{A}|\Re(\mu)|}, & \forall \nu \in i\R. 
	\end{array} \right. 
$$
Applying once again the Phragmen-Lindel\"of principle, we obtain
$$
  |F(\mu,\nu)| \leq C e^{\bar{A}|\Re(\mu)| + \bar{B}|\Re(\nu)|}, \quad \forall (\mu, \nu) \in \C^2, 
$$
which proves the Proposition. 
\end{proof}

\vspace{0.5cm}
We can now apply Theorem \ref{Be}. First define the analytic function
$$
  f(\mu,\nu) := F(\mu,\nu) e^{-\bar{A} \mu - \bar{B}\nu},
$$
where $\bar{A}$ and $\bar{B}$ are the positive constants appearing in Proposition \ref{EstF}. Then it is clear from Proposition \ref{EstF} that $f$ is bounded and analytic on the set 
$$
  T((\R^+)^2) = \{ (\mu,\nu) \in \C^2 \ \mid  \ \Re(\mu,\nu) \in (\R^+)^2 \}.
$$		
Second define the cone
\begin{equation} \label{Cone}
  \mathcal{C}_\epsilon = \{ (\mu, \theta \mu) \in (\R^+)^2 \ / \ \mu \in \R^+, \ \sqrt{c_1 + \e} \leq \theta \leq \sqrt{c_2 - \e} \}, \quad 0 < \e << 1, 
\end{equation}
where
\begin{equation} \label{c1c2}
  c_1 = \max \left( -\frac{s_{32}}{s_{33}} \right), \quad c_2 = \min \left( -\frac{s_{22}}{s_{23}} \right).
\end{equation}

\begin{rem}
The fact that $c_1 < c_2$ is ensured by Proposition \ref{StackelForm}. 
\end{rem}

Define also the discrete set 
\begin{equation} \label{E}
  E_M = \{ (\mu_m,\nu_m) \in (\R^+)^2 \ / \ m \geq M \},
\end{equation}
where $M$ is chosen large enough to ensure \footnote{It has been shown by Gobin \cite{Go2018}, Lemma 2.9., that there exist constants $C_1, C_2, D_1, D_2$ such that for all $m \geq 1$
$$
  C_1 \mu_m^2 + D_1 \leq \nu_m^2 \leq C_2 \mu_m^2 + D_2,
$$
where 
$$
  C_1 = \min \left( -\frac{s_{32}}{s_{33}} \right) > 0, \quad C_2 = \max \left( -\frac{s_{22}}{s_{23}} \right) > 0.
$$
This result implies the asserted claim.} that for all $m \geq M$, the joint spectrum $(\mu^2_m,\nu^2_m)$ of the angular operators $(H,L)$ belongs to $(\R^+)^2$. In that case, $(\mu_m,\nu_m)$ simply denotes the positive square root of $(\mu^2_m,\nu^2_m)$.

We now recall the following results shown by Gobin \cite{Go2018}, Appendices B and C. 

\begin{lemma} \label{Go}
1. There exists $h > 0$ such that 
$$
  |e_1 - e_2| \geq h, \quad \forall (e_1,e_2) \in (E_m \cap \mathcal{C}_\e)^2. 
$$	
2. Set $N(r) = \#  (E_m \cap \mathcal{C}_e) \cap B(0,r)$. Then
$$
 \varlimsup_{r \to \infty} \frac{N(r)}{r^2} > 0. 
$$
\end{lemma}

\begin{rems}
\textbf{1.} The proof of the second assertion in Lemma \ref{Go} is an application of the papers \cite{CdV1979, CdV1980} by Colin de Verdi\`ere in which the author studies basic properties of the joint spectrum of commuting pseudo-differential operators on manifolds (such as the density of the joint spectrum in certain cones which is needed here, or Bohr-Sommerfeld quantization formulas). 

\textbf{2.} In \cite{CdV1979}, the joint spectrum $J$ is counted with multiplicity and thus the density $N(r)$ estimated in Lemma \ref{Go} also counts this multiplicity. However, the discrete subset $E$ considered by Berndtsson in Theorem \ref{Be} does not count multiplicity and so for the density $n(r)$. Fortunately, we can show easily that each $(\mu_m^2,\nu_m^2)$ in the joint spectrum $J$ has at most multiplicity $4$ (see Gobin \cite{Go2018}, Remark 2.7). Hence the density $N(r)$ calculated with Colin de Verdi\`ere's work differs at most by a factor $4$ to the density $n(r)$ needed in Theorem \ref{Be}. In all cases, the density $n(r)$ remains of order $r^2$ and we can apply Berndtsson's Theorem \ref{Be}. 
\end{rems}

Hence applying Theorem \ref{Be} to the bounded and analytic function $f(\mu,\nu)$ on $T(\mathcal{C}_\e)$, we see that $f$ vanishes identically on $T(\mathcal{C}_\e)$ and thus on $\C^2$ by analytic continuation. Using the definition of $f$, we infer that the function $F(\mu,\nu)$ vanishes identically on $\C^2$ and by definition (\ref{F}) of $F$, this means that 
\begin{equation} \label{f6}
  M(\mu^2,\nu^2) = \tilde{M}(\mu^2,\nu^2), \quad \forall (\mu,\nu) \in \C^2 \backslash \{poles\}
\end{equation}

It remains now to extract the new information given by (\ref{f6}). Instead of working with the WT function $M(\mu^2,\nu^2)$ that depends on two variables, we prefer to work with the more classical WT function $M_{q_\nu}(\mu^2)$ defined in (\ref{NewWT}) that depends on one variable. Recalling the link between the two different WT functions given by (\ref{LinkM}) and using (\ref{s12}) - (\ref{s12'}), we infer from (\ref{f6}) that
\begin{equation} \label{f7}
  M_{q_\nu}(\mu^2) = \tilde{M}_{\tilde{q}_\nu}(\mu^2), \quad \forall (\mu,\nu) \in \C^2. 
\end{equation}
For each fixed $\nu \in \C$, we can apply the Borg-Marchenko Theorem (see for instance \cite{Be2001, Bo1946, Bo1952, GS2000, Te2014}) and obtain from (\ref{f7})
\begin{equation} \label{qnu=}
  \left\{ \begin{array}{rcl} 
	\bar{A} & = & \bar{\tilde{A}}, \\  
	q_\nu(u^1) & = & \tilde{q}_\nu(u^1), \quad \forall u^1 \in [0,\bar{A}], \ \forall \nu \in \C. 
  \end{array} \right. 
\end{equation}
Observe that working in the variable $u^1$, the cylinders $\Omega$ and $\tilde{\Omega}$ are exactly the same. Moreover, thanks to the definition (\ref{qnu-psi1}) of the potential $q_\nu$ and playing with two different values of $\nu \in \C$, we finally obtain
\begin{equation} \label{s13phi1=}
  \left\{ \begin{array}{rcl} 
	\bar{s}_{13}(u^1) & = & \bar{\tilde{s}}_{13}(u^1), \quad \forall u^1 \in [0,\bar{A}], \\  
	\bar{\phi}_1(u^1) & = & \bar{\tilde{\phi}}_1(u^1), \quad \forall u^1 \in [0,\bar{A}], 
  \end{array} \right. 
\end{equation}
where we recall that
$$
  \bar{s}_{13}(u^1) = \frac{s_{13}(x^1(u^1))}{s_{12}(x^1(u^1))}, \quad \bar{\phi}_1 = \frac{\phi_1}{s_{12}} - \frac{(\dot{\log s_{12}})^2}{16} + \frac{(\ddot{\log s_{12}})}{4}. 
$$

Let us pause a moment and see what information we have obtained exactly. It is easier to describe now our conformally St\"ackel manifolds $(M,G)$ and $(\tilde{M}, \tilde{G})$ in the unique coordinate system $(u^1,u^2,u^3)$ given by
\begin{equation} \label{U}
  u^1 = (\ref{u1}), \quad u^2 = x^2, \quad u^3 = x^3,
\end{equation}
and remember that the change of variable $u^1$ does not break the conformally St\"ackel structure of the metrics. In that case, we have shown (see (\ref{NewG}) - (\ref{NewPhii})) that
$$
  \Omega = \tilde{\Omega} = [0,\bar{A}] \times \mathbb{T}^2,
$$
and
\begin{equation} \label{g1}
  \bar{G} = \bar{c}^4 \bar{g}, \quad \bar{\tilde{G}} = \bar{\tilde{c}}^4 \bar{\tilde{g}}.  
\end{equation}
Here $\bar{g}, \, \bar{\tilde{g}}$ are St\"ackel metrics associated to the St\"ackel matrices 
\begin{equation} \label{g2}
  \bar{S}=\left(
\begin{matrix}
\bar{s}_{11}(u^{1})& 1 & \bar{s}_{13}(u^{1}) \\
s_{21}(u^{2})&s_{22}(u^{2}) & s_{23}(u^{2}) \\
s_{31}(u^{3})&s_{32}(u^{3})&s_{33}(u^{3})
\end{matrix}
\right)\,, 
\quad 
\bar{\tilde{S}}=\left(
\begin{matrix}
\bar{\tilde{s}}_{11}(u^{1})& 1 & \bar{s}_{13}(u^{1}) \\
\tilde{s}_{21}(u^{2})&s_{22}(u^{2}) & s_{23}(u^{2}) \\
\tilde{s}_{31}(u^{3})&s_{32}(u^{3})&s_{33}(u^{3})
\end{matrix}
\right)\,,
\end{equation}
thanks to (\ref{c8}) and (\ref{s13phi1=}). Moreover, the St\"ackel metric $\bar{g}$ has the local expression
\begin{equation} \label{g3}
  g = \sum_{i=1}^3 \bar{h}_i^2 (du^i)^2, \quad \bar{h}_i^2 = \frac{\det \bar{S}}{s_{12}}, \ i=1,2,3, 
\end{equation}
whereas the conformal factors $\bar{c},$ still satisfies the PDE (\ref{PDEc}) which has now the local expression
\begin{equation} \label{g4}
-\Delta_{\bar{g}} \bar{c} - \sum_{i=1}^3 \bar{h}_{i}^{2}\big( \bar{\phi}_{i} + \frac{1}{4} \bar{\gamma}_{i}^{2} - \frac{1}{2}\partial_{i} \bar{\gamma}_{i} \big) c = 0\,,
\end{equation}
with
\begin{equation} \label{g5}
  \bar{\phi}_1 = \frac{\phi_1}{s_{12}} - \frac{(\dot{\log s_{12}})^2}{16} + \frac{(\ddot{\log s_{12}})}{4}, \quad \bar{\phi}_2 = \phi_2, \quad \bar{\phi}_3 = \phi_3. 
\end{equation}

The crucial remark is that, \emph{in this coordinate system, the last two columns of $\bar{S}$ and $\bar{\tilde{S}}$ coincide}. We will use this observation in the next section to finish the proof of our uniqueness result for the inverse problem.

% THE CONFORMAL FACTOR

\subsection{The elliptic PDE on the conformal factor} \label{III4}

Working in the coordinate system $(u^1,u^2,u^3)$ given by (\ref{U}), we see that the metrics $\bar{G}$ can be written as
$$
  \bar{G} = \alpha \, g_0, \quad \alpha = \bar{c}^4 \det \bar{S}, \quad g_0 = \frac{1}{\bar{s}^{11}}(du^1)^2 + \frac{1}{\bar{s}^{21}}(du^2)^2 + \frac{1}{\bar{s}^{31}}(du^3)^2, 
$$
and a corresponding expression for $\bar{\tilde{G}}$ holds. Notice from (\ref{g2}) and (\ref{g4}) that
\begin{equation} \label{h1}
  g_0 = \tilde{g_0}.
\end{equation}
Thus it only remains to prove that $\alpha = \tilde{\alpha}$ in order to show that $\bar{G} = \bar{\tilde{G}}$. For this, we use

\begin{lemma}
The conformal factor $\alpha$ satisfies the elliptic PDE
\begin{equation} \label{h2}
  -\Delta_{g_0} \alpha - Q_{g_0,\bar{\phi}_i} \alpha = 0, 
\end{equation}
where
$$
  Q_{g_0,\bar{\phi}_i} =\sum_{i=1}^3 g_0^{ii} \left[ \frac{\partial^2_{ii} \log \det g_0}{4}  + \frac{\partial_{i} \log \det g_0}{8} + \frac{( \partial_{i} \log \det g_0)^2}{16} + \bar{\phi}_i \right].  
$$
\end{lemma}

\begin{proof}
We start from the PDE (\ref{g4}) satisfied by the conformal factor $\bar{c}$ and we recall that this PDE comes from the generalized Robertson-Condition (\ref{GenRob}), \textit{i.e.}
$$
  \sum_{i=1}^3 \bar{H}_i^{-2} \left( \frac{\partial_i \bar{\Gamma}_i}{2} - \frac{(\bar{\Gamma}_i)^2}{4} - \bar{\phi}_i \right) = 0. 
$$
Recalling that $\ds \bar{H}_i^2 = \frac{\alpha}{\bar{s^{i1}}}$, a direct calculation shows that $\alpha$ satisfies
\begin{equation} \label{h3}
	\sum_{i=1}^3 s^{i1} \left( -\partial_{ii}^2 \alpha - \frac{1}{2} \partial_i \alpha \right) + Q \alpha = 0, 
\end{equation}
where
\begin{equation} \label{h4}	
  Q = \sum_{i=1}^3 s^{i1} \left( \frac{ \partial_{ii}^2 \log s^{11} s^{21} s^{31}}{4} + \frac{ \partial_{i} \log s^{11} s^{21} s^{31}}{8} - \frac{ (\partial_{i} \log s^{11} s^{21} s^{31})^2}{16} - \bar{\phi}_i \right). 	
\end{equation}
We observe then that
$$
  s^{11} s^{21} s^{31} = \frac{1}{\det g_0}, \quad \sum_{i=1}^3 s^{i1} \left( -\partial_{ii}^2 \alpha - \frac{1}{2} \partial_i \alpha \right) = -\Delta_{g_0}. 
$$
Hence we obtain easily from (\ref{h3}) and (\ref{h4}) the elliptic PDE (\ref{h2}) satisfied by $\alpha$. 
\end{proof}

Thanks to (\ref{e12}), (\ref{s13phi1=}), (\ref{g5}) and (\ref{h1}), we see that the conformal factors $\alpha$ and $\tilde{\alpha}$ satisfy the \emph{same} second order elliptic PDE (\ref{h2}). Recalling that the last two columns of the St\"ackel matrices (\ref{g2}) coincide, it is easy to see from the definitions
$$
  \alpha = \bar{c}^4 \det \bar{S}, \quad \tilde{\alpha} = \bar{\tilde{c}}^4 \det \bar{\tilde{S}},
$$
and from the boundary determination results from Section \ref{III2} (see in particular (\ref{c13}) and (\ref{d2})), that $\alpha$ and $\tilde{\alpha}$ have the same Cauchy data at $u^1 = 0$, \textit{i.e.}
$$
  \alpha(0) = \tilde{\alpha}(0), \quad \dot{\alpha}(0) = \dot{\tilde{\alpha}}(0). 
$$
Hence, classical unique continuation (see for instance \cite{HoI2013}) gives
$$
  \alpha = \tilde{\alpha}, \quad \textrm{on} \ \Omega.
$$	
Consequently, we have shown
$$
  \bar{G} = \bar{\tilde{G}},
$$
or equivalently,
$$
  G = \tilde{G}, \quad \textrm{up to a change of variables of the form (\ref{CV})}.
$$
This finishes the proof of our uniqueness result in the anisotropic Calder\'on problem on three-dimensional conformally St\"ackel cylinders.

%%%%%%%%%%%%%%%%%%%%%%%%%%%%%%%%%%%%%%%%%%%% CONCLUSION %%%%%%%%%%%%%%%%%%%%%%%%%%%%%%%%%%%%%%%%%%%%%%%%%%%%%%%%%%%%%%

\Section{Some perspectives}

The results in this paper could be extended in several directions. \\

\noindent \textbf{1}. There exists a theory of \emph{non-orthogonal} St\"ackel manifolds (in the sense that the metrics are non-diagonal) for which the HJ and Helmholtz equations admit a complete set of classical separated solutions at all energies \cite{BCR2002a, BCR2002b, KM1983, KM1983b}. In particular, these non-orthogonal St\"ackel manifolds contain (and generalize enormously) the well-known family of Kerr black holes in General Relativity and their Riemannian counterparts. It would be interesting 

a) to extend this theory to the case of HJ and Helmholtz equation at fixed energy using the notion of $R$-separability following the lines of \cite{CR2006}, 

b) to address the question of uniqueness for the anisotropic Calder\'on problem in this non-orthogonal setting.  \\

\noindent \textbf{2}. The methods employed in this paper should work in more general situations in which the Laplace equation could be separated with respect to one variable only (and not all the variables as in the present case). Such models have been studied recently by us in \cite{DKN2019d} and named conformally Painlev\'e manifolds. This class of manifolds contains Riemannian manifolds of dimension $n$ for which the geodesic flow is not completely integrable, but rather possesses $1 \leq r < n-1$ \emph{hidden symmetries}, that is conformal Killing tensors of rank two satisfying certain additional assumptions. In such manifolds, the HJ and Laplace equations can be separated in groups of variables, leading to $r$ coupled PDEs. 

We intend in the near future to study the anisotropic Calder\'on problem on conformally Painlevé manifolds. \\

% QUANTUM COMPLETELY INTEGRABLE MANIFOLDS

%%%%%%%%%%%%%%%%%%%%%%%%%%%%%%%%%%%%%%%%%%%% BIBLIOGRAPHY %%%%%%%%%%%%%%%%%%%%%%%%%%%%%%%%%%%%%%%%%%%%%%%%%%%%%%%%%%%%%%%%

%\nocite{*}
\bibliographystyle{plain}
\bibliography{Biblio}

\noindent \footnotesize{DEPARTEMENT DE MATHEMATIQUES. UMR CNRS 8088. UNIVERSITE DE CERGY-PONTOISE. 95302 CERGY-PONTOISE. FRANCE. \\
\emph{Email adress}: thierry.daude@u-cergy.fr \\

\noindent DEPARTMENT OF MATHEMATICS AND STATISTICS. MCGill UNIVERSITY. MONTREAL, QC, H3A 2K6, CANADA \\
\emph{Email adress}: nkamran@math.mcgill.ca \\

\noindent LABORATOIRE DE MATHEMATIQUES JEAN LERAY. UMR CNRS 6629. 2 RUE DE LA HOUSSINIERE BP 92208. F-44322 NANTES CEDEX 03 \\
\emph{Email adress}: francois.nicoleau@univ-nantes.fr}

\end{document}